\documentclass[A4paper,11pt]{article}
\usepackage{amsfonts}
\usepackage{amsmath,amsthm}
\usepackage{amssymb,amscd}
\usepackage{blkarray}
\usepackage{enumerate}

\usepackage{kbordermatrix}
\usepackage{color}
\usepackage{cite}
\usepackage{physics}
\usepackage{esint}
\usepackage{tikz}
\usepackage{tikzscale}
\usepackage[absolute,overlay]{textpos}
\usepackage{multicol}
\usepackage{framed}
\usepackage{lscape}
\usepackage{caption}
\usepackage{subcaption}
\usepackage[inline]{enumitem}
\usepackage{float}
\usepackage{lipsum}
\usepackage{titlesec}
\usepackage{mathtools}

\titleformat{\paragraph}
{\normalfont\normalsize\bfseries}{\theparagraph}{1em}{}
\titlespacing*{\paragraph}
{0pt}{3.25ex plus 1ex minus .2ex}{1.5ex plus .2ex}
\renewcommand{\kbldelim}{(}
\renewcommand{\kbrdelim}{)}
\makeatletter

\newcommand{\Rmnum}[1]{\expandafter\@slowromancap\romannumeral #1@}
\makeatother

\def\XXint#1#2#3{{\setbox0=\hbox{$#1{#2#3}{\int}$}
\vcenter{\hbox{$#2#3$}}\kern-.5\wd0}}

\def\XXint#1#2#3{{\setbox0=\hbox{$#1{#2#3}{\int}$}
\vcenter{\hbox{$#2#3$}}\kern-.5\wd0}}

\theoremstyle{remark}

\theoremstyle{definition}
\newtheorem{theorem}{Theorem}[section]

\newtheorem{corollary}[theorem]{Corollary}
\newtheorem{lemma}[theorem]{Lemma}

\newtheorem{remark}{Remark}[section]

\newtheorem{proposition}[theorem]{Proposition}

\begin{document}
\title{Integrability of Multispecies Long-Range Swap Models with Species-Dependent Interpolation}
\author{\textbf{Eunghyun Lee\footnote{eunghyun.lee@nu.edu.kz}}\\ {\text{Department of Mathematics,}}
                                         \date{}   \\ {\text{School of Sciences and Humanities,}} \\ {\text{Nazarbayev University, }}\\ {\text{53 Kabanbay Batyr, Astana, Kazakhstan }}   }

\date{}
\maketitle
\begin{abstract}
\noindent
We introduce a class of multispecies exclusion processes with long-range swap interactions, incorporating species-dependent interpolation between TASEP-type and drop--push-type dynamics: each species $i$ is assigned a parameter $\mu_i$ governing same-species interactions, resulting in a heterogeneous system in which different species follow distinct microscopic interaction mechanisms. In contrast to previously studied integrable multispecies models, where species dependence typically enters through jump rates, the present framework allows the interaction mechanism itself to depend on the species. Our main result establishes integrability of the model in the binary parameter regime $\mu_i \in \{0,1\}$ for arbitrary species compositions. In the continuous parameter regime $\mu_i \in (0,1)$, we identify several nontrivial classes of species compositions for which integrability is preserved. We further extend the model to include bidirectional motion, going beyond totally asymmetric dynamics. Using the coordinate Bethe ansatz, we prove two-particle reducibility and derive the associated scattering matrix, which is shown to satisfy the Yang--Baxter equation. The resulting scattering matrix exhibits genuinely species-dependent diagonal entries.
\end{abstract}

\section{Introduction}\label{747pm819}

Integrable multispecies extensions of exactly solvable stochastic particle models have been studied extensively in recent years \cite{AritaKunibaSakaiSawabe2009,FerrariMartin2007,Kuan2020,Lee-Raimbekov-2025,Tracy-Widom-2013}. Such extensions often reveal richer structures and deeper connections to combinatorics, representation theory, probability, and statistical mechanics \cite{AyyerMandelshtamMartin2024,CorteelMandelshtamWilliams2022,Mandelshtam2025}. However, integrability in such systems is typically fragile: even seemingly mild modifications of local dynamics can destroy exact solvability.

In a recent work \cite{Lee2026}, we introduced the multispecies totally asymmetric simple exclusion process (TASEP) with long-range swap interactions. In this model, particles attempt to jump to the right and interact with particles along their path according to their relative strength: a jumping particle exchanges position with the first weaker particle encountered to its right (with holes regarded as the weakest), while stronger particles are bypassed. Although this induces effective long-range behavior, the dynamics can be implemented through local update rules. The model was shown to be integrable via the coordinate Bethe ansatz, with the associated scattering matrix satisfying the Yang--Baxter equation.

We adopt the same notion of integrability as in \cite{Lee2026}. A stochastic interacting particle system is said to be integrable if its $n$-particle dynamics can be reduced to two-particle interactions, and the associated scattering matrices satisfy the Yang--Baxter equation. This structure enables the application of the coordinate Bethe ansatz, yielding exact formulas for transition probabilities and providing a powerful framework for asymptotic analysis. See \cite{Lee-2020,NagaoSasamoto2004,Schutz-1997,Tracy-Widom-2008,TracyWidom2009} for results on the asymmetric simple exclusion process (ASEP) and related stochastic particle systems.

Many integrable multispecies particle models arise as degenerations of colored stochastic vertex models \cite{Aggarwal2017,AggarwalNicolettiPetrov2025,Ayyer-Kuniba-2025,Borodin-Bufetov,BorodinWheeler2022,Zhong2024}. However, it remains unclear whether the multispecies long-range swap model considered here fits into this framework. While it is natural to expect a connection to a degeneration of a colored six-vertex model, no such realization is currently known.

An important aspect of the definition of the multispecies long-range swap model in \cite{Lee2026} concerns the interpretation of “weaker” and “stronger” in the case of same-species encounters. When a particle attempts to jump into a site occupied by another particle of the same species, there are two natural conventions: the incoming particle may treat the resident particle either as weaker or as stronger. In the previous work \cite{Lee2026}, integrability was established under the convention that the incoming particle treats the resident particle as \emph{stronger}. This case was referred to as the drop-push-type multispecies long-range swap model. It was also observed that the alternative convention can be treated in a similar manner, and was called the TASEP-type multispecies long-range swap model \cite{Lee2026}.

Since both the TASEP-type and drop--push-type multispecies long-range swap models are known to be integrable \cite{Lee2026}, it is natural to ask whether an interpolation between these two classes remains integrable. The main contribution of this paper is to address this question in a substantially more general setting.

More precisely, we introduce a multispecies long-range swap model in which
each species is assigned its own interpolation parameter governing same-species
interactions. In contrast to previously studied integrable multispecies models,
where species dependence typically enters through jump rates
\cite{BorodinFerrariSasamoto2009,Lee-2021} while the underlying interaction
rules remain uniform across species, the present framework allows the
microscopic interaction mechanism itself to depend on the species. This leads
to a genuinely heterogeneous system in which different species follow distinct
local dynamics. 

A key novelty of the present work is that integrability persists in the species-dependent binary parameter setting $\mu_i \in \{0,1\}$ for arbitrary species compositions. Moreover, in the continuous parameter regime $\mu_i \in (0,1)$, we establish integrability for several nontrivial classes of species compositions. This demonstrates that integrability can be retained even when the interaction mechanism varies across species, while still admitting a Yang--Baxter structure.

In addition, we show that these integrability results extend to a bidirectional setting, allowing particles to move in both directions, a feature not present in \cite{Lee2026}.

\subsection*{Organization of the paper}

The remainder of the paper is organized as follows. In Section~\ref{definitionofmodel}, we introduce the model, and in Section~\ref{mainresults}, we state the main results.
Section~\ref{Section3} is devoted to establishing integrability of the model via the coordinate Bethe ansatz. We first present the two-particle interaction matrices and the conditions for the bidirectional generalization in Section~\ref{TwoParticle}. In Section~\ref{408ama}, we show that three-particle interactions can be reduced to two-particle interactions and describe the corresponding particle dynamics. Section~\ref{340pm412} extends this analysis to general $n$-particle systems, which constitutes a nontrivial generalization beyond the three-particle case. In Section~\ref{342pm412}, we derive the effective swap rates, and in Section~\ref{343pm412}, we verify that the associated scattering matrix satisfies the Yang--Baxter equation.
Finally, Section~\ref{Section4} concludes with a discussion of future research directions.

\section{Definition of Model and Main Results}\label{section2}

First, we fix the notation used throughout the paper. Our conventions follow those of \cite{Lee2026}.

\subsection{Configurations and notations}\label{subsection21}

We consider a system of $n$ particles on $\mathbb{Z}$ with species labels in $\{1,\dots,N\}$. A configuration of the system is represented by a pair
\[
(X,\pi) = (x_1,\dots,x_n;\, \pi_1\cdots \pi_n),
\]
where $X=(x_1,\dots,x_n)\in \mathbb{Z}^n$ satisfies $x_1<\cdots<x_n$ and represents the particle positions, and $\pi=\pi_1\cdots \pi_n \in \{1,\dots,N\}^n$ is a word encoding the species of the particles from left to right.

For an initial configuration $(Y,\nu)$ and a terminal configuration $(X,\pi)$ at time $t$, we denote the transition probability by
\[
P_{(Y,\nu)}(X,\pi; t).
\]
For fixed $X$ and $Y$, we define $\mathbf{P}_Y(X;t)$ to be the $N^n \times N^n$ matrix whose $(\pi,\nu)$-entry is $P_{(Y,\nu)}(X,\pi; t)$, where rows and columns are indexed by words $\pi,\nu \in \{1,\dots,N\}^n$, respectively.

When no confusion arises, we omit subscripts and simply write $P(X,\pi; t)$ and $\mathbf{P}(X;t)$. Moreover,
\[
\left(\frac{d}{dt}\mathbf{P}(X;t)\right)_{\pi,\nu}
:= \frac{d}{dt} P_{(Y,\nu)}(X,\pi; t).
\]

Throughout the paper,  $\mathbf{I}$ denotes the identity matrix, with the dimension understood from context.
\subsection{Model and Dynamics}\label{definitionofmodel}

We begin by recalling the model in the totally asymmetric setting introduced in \cite{Lee2026}.

\subsubsection{Multispecies long-range swap dynamics in \cite{Lee2026}}

Each particle is equipped with an exponential clock of rate $1$. When its clock rings, the particle attempts to jump to the neighboring site to its right.

When a particle of species $i$ at site $x$ attempts to move to site $x+1$, which is occupied by a particle of species $j$, the interaction is governed by the following rules:
\begin{itemize}
\item[(i)] $i<j$: The incoming particle $i$ jumps over the particle $j$ at $x+1$ and moves to $x+2$.
\item[(ii)] $i>j$: The incoming particle $i$ and the resident particle $j$ exchange their positions.
\item[(iii)] $i=j$: There are two possible conventions:
\begin{itemize}
  \item[$\bullet$] (TASEP type) The incoming particle $i$ and the resident particle $j$ exchange their positions, resulting in no change in the configuration.
  \item[$\bullet$] (drop--push type) The incoming particle $i$ jumps over the particle $j$ at $x+1$ and moves to $x+2$.
\end{itemize}
\end{itemize}
If the incoming particle lands on a site that remains occupied, the same rule is applied recursively. Likewise, if a resident particle is displaced to an already occupied site by the incoming particle, the same ``jumping over'' mechanism is applied recursively. This sequence of local two-particle interactions, resulting in an effective long-range swap, can be formalized using the hidden-state formulation.
\begin{remark}[\cite{Lee-2024,Lee2026}]
A hidden state is a \textit{temporary} configuration in which two particles occupy the same site, denoted by $(x,x,ij)$ introduced to describe two-particle interaction. These states do not belong to the state space of the Markov process and are resolved instantaneously according to the interaction rules. Within this framework, any effective long-range swap can be decomposed into a sequence of consecutive two-particle interactions occurring at hidden states; see Example 2.1 of \cite{Lee2026} for an illustration.
\end{remark}

\subsubsection{Interpolation between TASEP-type and drop--push-type dynamics}

We introduce an interpolation between the TASEP-type and drop--push-type swap dynamics, in which each species is assigned its own interpolation parameter, leading to species-dependent interaction rules.

Each particle of species $i$ at site $x$ is equipped with an exponential clock of rate $1$, and when its clock rings, it attempts to move to the site $x+1$.

If the target site is occupied by a particle of species $j$, the interaction is governed by the following rules:
\begin{itemize}
\item[(i)] If $j \neq i$, then the same swap rule as in \cite{Lee2026} is applied.

\item[(ii)] If $j = i$, then the interaction is resolved probabilistically as follows:
\begin{itemize}
\item[$\bullet$] with probability $\mu_i$, the incoming particle jumps over the resident particle and moves to $x+2$;
\item[$\bullet$] with probability $\lambda_i = 1 - \mu_i$, the two particles exchange their positions.
\end{itemize}
\end{itemize}
If a particle is displaced to a neighboring site that is still occupied, the same interaction rule is applied recursively, with the displaced particle treated as the new active particle. In particular,  when such a displacement induces motion to the left and the neighboring site is occupied by a particle of the same species, the interaction is resolved with the probabilities reversed: with probability $\lambda_i = 1 - \mu_i$, the active particle jumps over the resident particle to the left, and with probability $\mu_i$, the two particles exchange their positions.

This choice of probabilities is consistent with both the interpolation structure of the model and the interpretation of the diagonal entries in (3) of \cite{Lee2026}.
More precisely, in the two-particle case, the $(11,11)$ entry (equal to 1) of the first matrix in (3) of \cite{Lee2026} represents the probability that a particle of species $1$ jumps over another particle of the same species when moving to the right. By symmetry, this is also the probability of exchanging positions under leftward motion.
Conversely, the $(11,11)$ entry of the second matrix represents the probability of exchanging positions under rightward motion, and, correspondingly, the probability of jumping over under leftward motion.

The cases $\mu_i = 0$ and $\mu_i = 1$ for all species $i$ correspond to the TASEP-type and drop--push-type dynamics, respectively, studied in \cite{Lee2026}.

\begin{figure}[H]
\centering
\begin{tikzpicture}[scale=1.0, every node/.style={font=\small}]

    \draw[thick] (0,0) -- (6,0);
    \foreach \x in {0,1,2,3,4,5,6}
        \draw (\x,0.08) -- (\x,-0.08);

    \node at (1,0.45) {$0$};
    \node at (2,0.45) {$i$};
    \node at (3,0.45) {$i$};
    \node at (4,0.45) {$0$};
    \node at (5,0.45) {$0$};

    \node at (3,-0.9) {same-species encounter};

    \draw[->,thick]
        (2,0.7) to[out=80,in=100] (3,0.7);

    \node at (2.5,1.3) {jump attempt};

    \fill (6.4,0) circle (1.2pt);

    \draw[->,thick] (6.4,0) -- (9.8,1.9);
    \draw[->,thick] (6.4,0) -- (9.8,-1.9);

    \node at (8.0,2.5) {$\mu_i$};
    \node at (8.0,-2.5) {$1-\mu_i$};

    \draw[thick] (10.2,1.9) -- (16.2,1.9);
    \foreach \x in {10.2,11.2,12.2,13.2,14.2,15.2,16.2}
        \draw (\x,1.98) -- (\x,1.82);

    \node at (11.2,2.35) {$0$};
    \node at (12.2,2.35) {$0$};
    \node at (13.2,2.35) {$i$};
    \node at (14.2,2.35) {$i$};
    \node at (15.2,2.35) {$0$};

    \node at (13.2,1.0) {jump over};

    \draw[thick] (10.2,-1.9) -- (16.2,-1.9);
    \foreach \x in {10.2,11.2,12.2,13.2,14.2,15.2,16.2}
        \draw (\x,-1.82) -- (\x,-1.98);

    \node at (11.2,-1.45) {$0$};
    \node at (12.2,-1.45) {$i$};
    \node at (13.2,-1.45) {$i$};
    \node at (14.2,-1.45) {$0$};
    \node at (15.2,-1.45) {$0$};

    \node at (13.2,-2.7) {no change};

\end{tikzpicture}
\caption{Same-species interaction. In the right-moving setting, with probability $\mu_i$, the active particle jumps over the adjacent particle of the same species, resulting in a shift of the configuration. With probability $1-\mu_i$, the two particles exchange positions.}
\label{fig:same-species}
\end{figure}
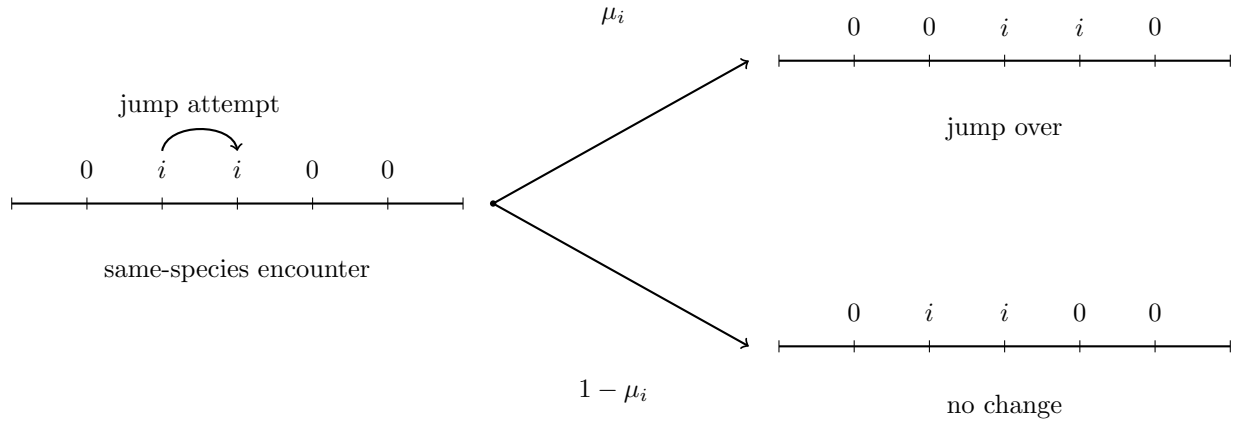

\subsubsection{Bidirectional generalization}\label{subsubsection11}

We further extend the model in \cite{Lee2026} by allowing particles to move in both directions.

Each particle of species $i$ is equipped with two independent exponential clocks: a rightward clock with rate $p$ and a leftward clock with rate $q = 1 - p$. When the rightward clock rings, the particle attempts to move from site $x$ to $x+1$; when the leftward clock rings, it attempts to move from $x$ to $x-1$.

The interaction rules for rightward moves are as defined in the previous subsection. For leftward moves, the rules are defined analogously, with the direction of motion reversed and the roles of $\mu_i$ and $\lambda_i$ interchanged. Specifically, when a particle of species $i$ at site $x$ attempts to move to $x-1$ and the target site is occupied by a particle of species $j$, the same swap rule applies if $j \neq i$. If $j = i$, the interaction is resolved with probability $\lambda_i$ by jumping over the resident particle, and with probability $\mu_i$ by exchanging positions. Repeated application of these local rules produces effective long-range swaps in both directions. It is shown that this choice of asymmetric probabilities enables the model to be integrable.

\subsection{Main results}\label{mainresults}

We summarize the main results for the bidirectional multispecies long-range swap model with species-dependent interpolation parameters $\mu_i$ defined above. These results establish that the Yang--Baxter structure is robust under species-dependent interaction mechanisms and that integrability persists in a broad class of cases.

Establishing integrability via the coordinate Bethe ansatz requires two ingredients:
(i) the reduction of the many-body dynamics to two-particle interactions, and
(ii) the Yang--Baxter equation for the associated two-particle scattering matrices.

In the present model, the Yang--Baxter equation is shown to hold for arbitrary species compositions and parameters $\mu_i \in [0,1]$. The reduction to two-particle interactions is governed by the invertibility of certain operators arising in the reduction procedure.

\begin{theorem}[Yang--Baxter equation]\label{356pam}
For an $n$-particle system with arbitrary species composition and species-dependent parameters $\mu_i \in [0,1]$, the associated two-particle scattering matrices satisfy the Yang--Baxter equation.
\end{theorem}

\begin{proposition}[Integrability under invertibility]
For multisets $M$ of species labels in the system for which the operators arising in the two-particle reduction procedure are invertible, the system admits a reduction of many-body dynamics to two-particle interactions.
\end{proposition}

\begin{theorem}[Binary parameter regime]\label{binarycase}
Suppose $\mu_i \in \{0,1\}$ for each species $i$. Then the associated operators are invertible for all multisets $M$ of species labels. Consequently, the model is integrable for arbitrary species compositions in this regime.
\end{theorem}

For $\mu_i \in (0,1)$, we establish this invertibility in several nontrivial classes of species compositions.

\begin{theorem}[Verified cases of invertibility]\label{757pm}
The invertibility of the associated operators holds for the following multisets $M$:
\begin{itemize}
\item[(i)] $M=[\nu_1,\dots,\nu_n]$ with $\nu_1=\cdots=\nu_n$;
\item[(ii)] $M=[\nu_1,\dots,\nu_n]$ with all $\nu_i$ distinct;
\item[(iii)] $M=[\nu_1,\dots,\nu_n]$ with $\nu_1 \neq \nu_2=\cdots=\nu_n$.
\end{itemize}
\end{theorem}

The invertibility of the general case ($\mu_i \in (0,1)$ with arbitrary species composition) remains open due to the complexity of the operators arising in the reduction procedure. Numerical evidence for small systems suggests that the invertibility condition may hold more broadly.

\section{Integrability}\label{Section3}
\subsection{Two-particle system}\label{TwoParticle}
We first examine the two-particle system, since it already captures the full local interaction structure of the model. In particular, both the species-dependent same-species interaction and the bidirectional dynamics are encoded at this level.

To identify the structural constraints required for integrability in the bidirectional setting, we first consider a general form of the local interaction rules.
\begin{itemize}
\item [(i)] When a particle of species $i$ at site $x$ attempts to jump to $x+1$, which is occupied by another particle of the same species, the interaction is resolved probabilistically:
\begin{itemize}
\item[$\bullet$] with probability $\mu^r_i$, the incoming particle jumps over the resident particle and moves to $x+2$;
\item[$\bullet$] with probability $\lambda^r_i = 1-\mu^r_i$, the incoming particle and the resident particle exchange positions.
\end{itemize}

\item [(ii)] When a particle of species $i$ at site $x$ attempts to jump to $x-1$, which is occupied by another particle of the same species, the interaction is resolved probabilistically:
\begin{itemize}
\item[$\bullet$] with probability $\mu^l_i$, the incoming particle jumps over the resident particle and moves to $x-2$;
\item[$\bullet$] with probability $\lambda^l_i = 1-\mu^l_i$, the incoming particle and the resident particle exchange positions.
\end{itemize}
\end{itemize}
These interaction rules can be encoded in matrix form, which will be convenient for the formulation of the master equation and the subsequent analysis.

Let
\[
\mathbf{M}^r =
\begin{pmatrix}
\mu^r_1 & 0 & 0 & 0 \\
0 & 0 & 0 & 0 \\
0 & 1 & 0 & 0 \\
0 & 0 & 0 & \mu^r_2
\end{pmatrix},\quad
\mathbf{N}^r =
\begin{pmatrix}
\lambda^r_1 & 0 & 0 & 0 \\
0 & 0 & 1 & 0 \\
0 & 0 & 0 & 0 \\
0 & 0 & 0 & \lambda^r_2
\end{pmatrix},
\]
be matrices encoding the local two-particle interactions corresponding to rightward \textit{jumping over} and \textit{swapping positions}, respectively, and
\[
\mathbf{M}^l =
\begin{pmatrix}
\mu^l_1 & 0 & 0 & 0 \\
0 & 0 & 1 & 0 \\
0 & 0 & 0 & 0 \\
0 & 0 & 0 & \mu^l_2
\end{pmatrix},\quad
\mathbf{N}^l =
\begin{pmatrix}
\lambda^l_1 & 0 & 0 & 0 \\
0 & 0 & 0 & 0 \\
0 & 1 & 0 & 0 \\
0 & 0 & 0 & \lambda^l_2
\end{pmatrix}
\]
be matrices corresponding to leftward \textit{jumping over} and \textit{swapping positions}, respectively.

Let $\mathbf{U}(x_1,x_2;t)$ be an $N^2 \times N^2$ matrix-valued function whose entries coincide with transition probabilities for $x_1<x_2$, and which is extended to all $(x_1,x_2)\in\mathbb{Z}^2$ via boundary relations.

The evolution of $\mathbf{U}(x_1,x_2;t)$ is governed by the master equation
\begin{equation}\label{1005pm810a}
\begin{aligned}
\frac{d}{dt}\mathbf{U}(x_1,x_2;t)
&= p\,\mathbf{U}(x_1-1,x_2;t)
   + p\,\mathbf{U}(x_1,x_2-1;t) + q\,\mathbf{U}(x_1+1,x_2;t)
   + q\,\mathbf{U}(x_1,x_2+1;t) \\
&\quad - 2\,\mathbf{U}(x_1,x_2;t),
\end{aligned}
\end{equation}
when $x_2 > x_1+1$, and by
\begin{equation}\label{931pm810b}
\begin{aligned}
\frac{d}{dt}\mathbf{U}(x,x+1;t)
&= p\Bigl[
    \mathbf{U}(x-1,x+1;t)
    + \mathbf{M}^r\,\mathbf{U}(x-1,x;t)
    + \mathbf{N}^r\,\mathbf{U}(x,x+1;t)
\Bigr] \\
&\quad + q\Bigl[
    \mathbf{U}(x,x+2;t)
    + \mathbf{M}^l\,\mathbf{U}(x+1,x+2;t)
    + \mathbf{N}^l\,\mathbf{U}(x,x+1;t)
\Bigr] \\
&\quad - 2\,\mathbf{U}(x,x+1;t),
\end{aligned}
\end{equation}
when $x_2 = x_1+1$.

Equations \eqref{1005pm810a} and \eqref{931pm810b} can be unified into a single equation on $\mathbb{Z}^2$, namely \eqref{1005pm810a}, together with the boundary condition
\begin{equation}\label{209am324}
\begin{aligned}
&p\Bigl[
   \mathbf{M}^r\,\mathbf{U}(x-1,x;t)
    + \mathbf{N}^r\,\mathbf{U}(x,x+1;t)
\Bigr] + q\Bigl[
    \mathbf{M}^l\,\mathbf{U}(x+1,x+2;t)
    + \mathbf{N}^l\,\mathbf{U}(x,x+1;t)
\Bigr] \\
&= p\,\mathbf{U}(x,x;t) + q\,\mathbf{U}(x+1,x+1;t).
\end{aligned}
\end{equation}

To ensure that the boundary condition can be written in a unified form, independent of the direction of motion, we impose the relations
\begin{equation}\label{902pm}
\mu^l_i = \lambda^r_i =: \lambda_i, \qquad \mu^r_i = \lambda^l_i =: \mu_i,
\end{equation}
so that $\mathbf{M}^r = \mathbf{N}^l =: \mathbf{B}$ and $\mathbf{N}^r = \mathbf{M}^l =: \mathbf{B}'$.

Under this identification, the raw boundary relation \eqref{209am324} becomes
\begin{equation}\label{209am324b}
\begin{aligned}
&p\Bigl[
   \mathbf{B}\,\mathbf{U}(x-1,x;t)
    + \mathbf{B}'\,\mathbf{U}(x,x+1;t)
\Bigr]  + q\Bigl[
    \mathbf{B}'\,\mathbf{U}(x+1,x+2;t)
    + \mathbf{B}\,\mathbf{U}(x,x+1;t)
\Bigr] \\
&= p\,\mathbf{U}(x,x;t) + q\,\mathbf{U}(x+1,x+1;t).
\end{aligned}
\end{equation}
We impose
\begin{equation}\label{212am324}
\mathbf{U}(x,x;t)
=
\mathbf{B}\,\mathbf{U}(x-1,x;t)
+
\mathbf{B}'\,\mathbf{U}(x,x+1;t)
\end{equation}
as a sufficient boundary condition. Then \eqref{212am324}, together with its shifted version at $x+1$ (i.e., \eqref{212am324} with $x$ replaced by $x+1$), implies \eqref{209am324b}.

Hence, we take \eqref{212am324} as the defining boundary condition of the model. It has the same structural form as equation~(6) in \cite{Lee2026}, with the interaction matrices replaced by their generalized counterparts, and extends the boundary conditions appearing in the single-species models of \cite{Ali,Ali2}. Moreover, the identification \eqref{902pm} is consistent with the defining rules given in Section~\ref{subsubsection11}.

\subsection{Three-particle system: Two-particle reducibility}\label{408ama}

In this section, we demonstrate a key structural feature underlying integrability, the \emph{two-particle reducibility}: the interaction of three particles can be described entirely in terms of successive two-particle interactions. The three-particle analysis captures the essential mechanism of the general $n$-particle reduction and will serve as a prototype for the arguments in the next section.

Although the boundary condition \eqref{212am324} has the same structural form as in \cite{Lee2026} and the overall strategy for establishing integrability is similar, the algebraic analysis in the present model is substantially more involved. In \cite{Lee2026}, the matrices $\mathbf{B}$ and $\mathbf{B}'$ give rise to nilpotent operators, which play a crucial role in establishing reducibility. In contrast, due to the presence of $\mu_i$ and $\lambda_i$ in $\mathbf{B}$ and $\mathbf{B}'$ in \eqref{212am324}, such nilpotent structures no longer persist in general.

We consider a three-particle system in which each particle has species in $\{1,\dots,N\}$. As an extension of the $4\times4$ matrices $\mathbf{B}$ and $\mathbf{B}'$ from the two-particle case, we define the $N^2 \times N^2$ matrices $\mathbf{B}=(b_{\pi,\nu})$ and $\mathbf{B}'=(b'_{\pi,\nu})$. The rows and columns are indexed lexicographically by words $\pi=\pi_1\pi_2$ and $\nu=\nu_1\nu_2$ ranging from $11$ to $NN$, and the entries are given by
\begin{equation}\label{138am42}
\begin{aligned}
b_{\pi,\nu} &=
\begin{cases}
\mu_i & \text{if } \pi=\nu=ii,\\[4pt]
1 & \text{if } \pi_1=\nu_2,\ \pi_2=\nu_1,\ \text{and } \nu_1<\nu_2,\\[4pt]
0 & \text{otherwise},
\end{cases}
\\[6pt]
b'_{\pi,\nu} &=
\begin{cases}
\lambda_i=1-\mu_i & \text{if } \pi=\nu=ii,\\[4pt]
1 & \text{if } \pi_1=\nu_2,\ \pi_2=\nu_1,\ \text{and } \nu_1>\nu_2,\\[4pt]
0 & \text{otherwise}.
\end{cases}
\end{aligned}
\end{equation}

Let $\mathbf{U}(x_1,x_2,x_3;t)$ be an $N^3 \times N^3$ matrix-valued function whose entries represent transition probabilities for $x_1<x_2<x_3$, and extend it to all $(x_1,x_2,x_3)\in\mathbb{Z}^3$ via boundary conditions which will be given below.

The master equation is
\begin{equation}\label{1055pm325}
\frac{d}{dt}\mathbf U(x_1,x_2,x_3;t)
=
\sum_{i=1}^3 \big( p\mathcal{T}_i^- + q\mathcal{T}_i^+ - 1 \big)\mathbf U(x_1,x_2,x_3;t),
\end{equation}
where $\mathcal{T}_i^{\pm}$ shifts the $i$-th spatial coordinate by $\pm1$, and the boundary conditions are the natural extensions of \eqref{212am324}:
\begin{align}
\mathbf{U}(x,x,x';t)
&= (\mathbf{B} \otimes \mathbf{I})\,\mathbf{U}(x-1,x,x';t)
+ (\mathbf{B}' \otimes \mathbf{I})\,\mathbf{U}(x,x+1,x';t), \label{1154am8121}\\[6pt]
\mathbf{U}(x',x,x;t)
&= (\mathbf{I} \otimes \mathbf{B})\,\mathbf{U}(x',x-1,x;t)
+ (\mathbf{I} \otimes \mathbf{B}')\,\mathbf{U}(x',x,x+1;t), \label{1155am8121}
\end{align}
for all $x,x'\in\mathbb{Z}$ where $\mathbf{I}$ denote the $N \times N$ identity matrix.

The master equation \eqref{1055pm325}, together with the boundary conditions \eqref{1154am8121}–\eqref{1155am8121}, induces evolution equations that depend on the relative positions of the particles. In particular, when two particles are adjacent while the remaining particle is separated, namely in the regions
\begin{equation}
x_1 = x_2 - 1 < x_3 - 1, \qquad
x_1 < x_2 - 1 = x_3 - 1,
\end{equation}
the evolution equation involves terms corresponding to non-admissible configurations,
\[
(x,x,x') \quad (x < x'-1), \qquad
(x',x,x) \quad (x' < x-1),
\]
in which two particles occupy the same site. These configurations are eliminated using the boundary conditions \eqref{1154am8121} and \eqref{1155am8121}.

When all three particles are adjacent, that is, $x_1 = x_2 - 1 = x_3 - 2$,
additional nontrivial contributions arise due to the interaction of overlapping boundary terms, reflecting the simultaneous interaction of all three particles.

\subsubsection{Elimination of non-admissible configurations}\label{302am1}

In this section, we express the non-admissible terms
\[
\mathbf U(x,x,x+1;t), \quad \mathbf U(x,x+1,x+1;t),
\]
which arise in the master equation when all three particles are adjacent, in terms of admissible configurations using $\mathbf B$ and $\mathbf B'$.

From \eqref{1154am8121} and \eqref{1155am8121}, we obtain the coupled system for the non-admissible configurations:
\begin{align}
\mathbf U(x,x,x+1;t)
&=
(\mathbf B \otimes \mathbf I)\,\mathbf U(x-1,x,x+1;t)
+
(\mathbf B' \otimes \mathbf I)\,\mathbf U(x,x+1,x+1;t),
\label{eq:red1-comb}\\
\mathbf U(x,x+1,x+1;t)
&=
(\mathbf I \otimes \mathbf B)\,\mathbf U(x,x,x+1;t)
+
(\mathbf I \otimes \mathbf B')\,\mathbf U(x,x+1,x+2;t).
\label{eq:red2-comb}
\end{align}

This system can be solved by eliminating one of the non-admissible variables.

Substituting \eqref{eq:red1-comb} into \eqref{eq:red2-comb} yields
\begin{equation}\label{eq:redX}
\begin{aligned}
(\mathbf I - \mathbf X)\,\mathbf U(x,x+1,x+1;t)
&=
(\mathbf I \otimes \mathbf B)(\mathbf B \otimes \mathbf I)\,
\mathbf U(x-1,x,x+1;t) \\
&\quad
+
(\mathbf I \otimes \mathbf B')\,\mathbf U(x,x+1,x+2;t),
\end{aligned}
\end{equation}
where $\mathbf X := (\mathbf I \otimes \mathbf B)(\mathbf B' \otimes \mathbf I)$.
Provided that $\mathbf I-\mathbf X$ is invertible (a property that will be verified in Section~\ref{302am2}), we obtain
\begin{equation}\label{eq:redX-sol}
\begin{split}
\mathbf U(x,x+1,x+1;t)
&=
(\mathbf I-\mathbf X)^{-1}
(\mathbf I \otimes \mathbf B)(\mathbf B \otimes \mathbf I)\,\mathbf U(x-1,x,x+1;t)\\
&\quad
+
(\mathbf I-\mathbf X)^{-1}
(\mathbf I \otimes \mathbf B')\,\mathbf U(x,x+1,x+2;t).
\end{split}
\end{equation}
Substituting \eqref{eq:redX-sol} into \eqref{eq:red1-comb} then expresses $\mathbf U(x,x,x+1;t)$ entirely in terms of admissible configurations.

Alternatively, substituting \eqref{eq:red2-comb} into \eqref{eq:red1-comb} gives
\begin{equation}\label{eq:redY}
\begin{aligned}
(\mathbf I - \mathbf Y)\,\mathbf U(x,x,x+1;t)
&=
(\mathbf B \otimes \mathbf I)\,\mathbf U(x-1,x,x+1;t) \\
&\quad
+
(\mathbf B' \otimes \mathbf I)(\mathbf I \otimes \mathbf B')\,
\mathbf U(x,x+1,x+2;t),
\end{aligned}
\end{equation}
where $\mathbf Y := (\mathbf B' \otimes \mathbf I)(\mathbf I \otimes \mathbf B)$.

Provided that $\mathbf I-\mathbf Y$ is invertible, this yields an alternative expression for $\mathbf U(x,x,x+1;t)$:
\begin{equation}\label{eq:redY-sol}
\begin{split}
\mathbf U(x,x,x+1;t)
&=
(\mathbf I-\mathbf Y)^{-1}
(\mathbf B \otimes \mathbf I)\,\mathbf U(x-1,x,x+1;t)\\
&\quad
+
(\mathbf I-\mathbf Y)^{-1}
(\mathbf B' \otimes \mathbf I)(\mathbf I \otimes \mathbf B')\,\mathbf U(x,x+1,x+2;t).
\end{split}
\end{equation}
Substituting \eqref{eq:redY-sol} into \eqref{eq:red2-comb} expresses $\mathbf U(x,x+1,x+1;t)$ in terms of admissible configurations. Both elimination procedures yield equivalent expressions, showing that all non-admissible configurations can be expressed by admissible ones.
\subsubsection{Invertibility of $(\mathbf I-\mathbf X)$ and $(\mathbf I-\mathbf Y)$}\label{302am2}
In this section, we show the invertibility of $(\mathbf I-\mathbf X)$ and $(\mathbf I-\mathbf Y)$. The analysis of the two operators is parallel.

Note that \(\mathbf B'\otimes \mathbf I\) describes the interaction in which the incoming particle pushes the resident particle backward, acting on the first and second particles, which results in swapping positions. On the other hand, \(\mathbf I\otimes \mathbf B\) describes the interaction in which the incoming particle jumps over the resident particle, acting on the second and third particles.

It is convenient to describe the action of these operators on the canonical basis using bra--ket notation. For a word \(\nu=\nu_1\nu_2\nu_3\in\{1,\dots,N\}^3\), we write $
|\nu\rangle = |\nu_1\nu_2\nu_3\rangle $ for the corresponding basis vector in \(\mathbb C^{N^3}\). For two words \(\pi,\nu\), the matrix element of an operator \(\mathbf A\) is written as $\langle \pi \mid \mathbf A \mid \nu \rangle,$ which coincides with the \((\pi,\nu)\)-entry of \(\mathbf A\) under lexicographic ordering. In this notation, the action of \(\mathbf A\) on basis vectors is given by
\[
\mathbf A|\nu\rangle = \sum_{\pi}\langle \pi \mid \mathbf A \mid \nu\rangle\,|\pi\rangle.
\]

In the special case $\mu_i=1$ for all $i$, it was shown in \cite{Lee2026} that $\mathbf X$ is nilpotent, which implies the invertibility of $(\mathbf I-\mathbf X)$. In the present general setting $\mu_i \in [0,1]$, $\mathbf X$ is no longer nilpotent on the whole space. However, Lemma~\ref{lem:X-structure} below shows that $\mathbf X$ is nilpotent on the complement of the single-species sector and acts on each vector $|iii\rangle$ by the scalar $\mu_i\lambda_i$. Since $0\le \mu_i\lambda_i<1$, this still ensures the invertibility of $(\mathbf I-\mathbf X)$.
\begin{lemma}\label{lem:X-structure}
Let
\[
\mathbf X := (\mathbf I\otimes \mathbf B)(\mathbf B'\otimes \mathbf I).
\]
Then the following hold.
\begin{itemize}
\item[(i)] The action of \(\mathbf X\) on the basis vector \(|\nu_1\nu_2\nu_3\rangle\) is given by
\[
\mathbf X|\nu_1\nu_2\nu_3\rangle
=
\begin{cases}
\lambda_{\nu_1}\mu_{\nu_1}\,|\nu_1\nu_1\nu_1\rangle, & \nu_1=\nu_2=\nu_3, \\[4pt]
\lambda_{\nu_1}\,|\nu_1\nu_3\nu_1\rangle, & \nu_1=\nu_2<\nu_3, \\[4pt]
\mu_{\nu_1}\,|\nu_2\nu_1\nu_1\rangle, & \nu_1>\nu_2,\ \nu_1=\nu_3, \\[4pt]
|\nu_2\nu_3\nu_1\rangle, & \nu_3>\nu_1>\nu_2, \\[4pt]
0, & \text{otherwise}.
\end{cases}
\]

\item[(ii)] For every \(k\ge2\),
\[
\mathbf X^k=\sum_{i=1}^N (\mu_i\lambda_i)^k \mathbf E_i,~~~\textrm{where}~\mathbf E_i=|iii\rangle\langle iii|.
\]
\end{itemize}
\end{lemma}

\begin{proof}
Part (i) follows by direct computation: first apply $\mathbf B'\otimes \mathbf I$ to the first two coordinates, and then apply $\mathbf I\otimes \mathbf B$ to the last two coordinates. For (ii), observe that if $|\nu_1\nu_2\nu_3\rangle \neq |iii\rangle$ for all $i$, then either $\mathbf X|\nu_1\nu_2\nu_3\rangle=0$, or $\mathbf X|\nu_1\nu_2\nu_3\rangle$ produces a vector whose first two coordinates are strictly increasing. In either case, applying $\mathbf B'\otimes \mathbf I$ once more yields zero, and hence
\[
\mathbf X^2|\nu_1\nu_2\nu_3\rangle=0.
\]
On the other hand, for $|iii\rangle$, we have
\[
\mathbf X|iii\rangle = \mu_i\lambda_i\,|iii\rangle,
\]
and thus
\[
\mathbf X^k|iii\rangle = (\mu_i\lambda_i)^k |iii\rangle
\quad \text{for all } k\ge1.
\]
Combining these observations gives the result.
\end{proof}

An analogous statement holds for $\mathbf Y$.

\begin{lemma}\label{lem:Y-structure}
Let
\[
\mathbf Y := (\mathbf B' \otimes \mathbf I)(\mathbf I \otimes \mathbf B).
\]
Then the following hold.
\begin{itemize}
\item[(i)] The action of \(\mathbf Y\) on the basis vector \(|\nu_1\nu_2\nu_3\rangle\) is given by
\[
\mathbf Y|\nu_1\nu_2\nu_3\rangle
=
\begin{cases}
\lambda_{\nu_1}\mu_{\nu_1}\,|\nu_1\nu_1\nu_1\rangle, & \nu_1=\nu_2=\nu_3,\\[4pt]
\mu_{\nu_2}\,|\nu_2\nu_1\nu_2\rangle, & \nu_1>\nu_2=\nu_3,\\[4pt]
\lambda_{\nu_1}\,|\nu_1\nu_1\nu_2\rangle, & \nu_1=\nu_3>\nu_2,\\[4pt]
|\nu_3\nu_1\nu_2\rangle, & \nu_1>\nu_3>\nu_2,\\[4pt]
0, & \text{otherwise}.
\end{cases}
\]

\item[(ii)] For every \(k\ge2\),
\[
\mathbf Y^k=\sum_{i=1}^N (\mu_i\lambda_i)^k \mathbf E_i,~~\textrm{where}~~ \mathbf E_i=|iii\rangle\langle iii|.
\]
\end{itemize}
\end{lemma}

\begin{proposition}\label{1259pm}
Let
\[
\mathbf X := (\mathbf I\otimes \mathbf B)(\mathbf B'\otimes \mathbf I).
\]
Then
\begin{equation}\label{1148pm325}
(\mathbf I-\mathbf X)^{-1}
=
\mathbf I
+\mathbf X
+\sum_{i=1}^N
\frac{(\mu_i\lambda_i)^2}{1-\mu_i\lambda_i}\,\mathbf E_i,
\qquad
\mathbf E_i = |iii\rangle\langle iii|.
\end{equation}
\end{proposition}

\begin{proof}
By Lemma~\ref{lem:X-structure}, the spectrum of $\mathbf X$ consists of $0$ and $\mu_i\lambda_i$ for $i=1,\dots,N$. Since $\lambda_i=1-\mu_i$ with $\mu_i\in[0,1]$, we have
\[
0 \le \mu_i\lambda_i \le \frac14 < 1.
\]
Hence the spectral radius of $\mathbf X$ is strictly less than $1$, and therefore
\[
(\mathbf I-\mathbf X)^{-1}
= \sum_{k=0}^\infty \mathbf X^k.
\]

Using Lemma~\ref{lem:X-structure},
\[
\mathbf X^k=\sum_{i=1}^N (\mu_i\lambda_i)^k \mathbf E_i
\quad \text{for } k\ge 2,
\]
we obtain
\[
(\mathbf I-\mathbf X)^{-1}
=
\mathbf I + \mathbf X
+\sum_{i=1}^N \sum_{k=2}^\infty (\mu_i\lambda_i)^k \mathbf E_i,
\]
which yields \eqref{1148pm325}.
\end{proof}

Analogously, $(\mathbf I-\mathbf Y)$ is invertible and admits the same explicit form.

\subsubsection{Dynamics of the three-particle interaction}\label{dynamics3}
In this subsection, let us investigate the dynamics of the three-particle interaction. 
We consider configurations in which three particles occupy adjacent sites.
After eliminating all non-admissible spatial configurations using the results of the previous section, the master equation for
\(\mathbf{U}(x,x+1,x+2;t)\) takes the form
\begin{equation}\label{1149pm}
\begin{aligned}
\frac{d}{dt}\mathbf U(x,x+1,x+2;t)
&=
p\,\mathbf U(x-1,x+1,x+2;t)
+ q\,\mathbf U(x,x+1,x+3;t) \\[4pt]
&\quad
+ p\,(\mathbf B\otimes \mathbf I)\,\mathbf U(x-1,x,x+2;t)
+ q\,(\mathbf I\otimes \mathbf B')\,\mathbf U(x,x+2,x+3;t) \\[8pt]
&\quad
+ p\,(\mathbf B'\otimes \mathbf I)\,\mathbf U(x,x+1,x+2;t)
+ q\,(\mathbf I\otimes \mathbf B)\,\mathbf U(x,x+1,x+2;t) \\[4pt]
&\quad
+ p\,(\mathbf I-\mathbf X)^{-1}
\Big[
(\mathbf I\otimes \mathbf B)(\mathbf B\otimes \mathbf I)\,
\mathbf U(x-1,x,x+1;t)
\Big] \\[4pt]
&\quad
+ p\,(\mathbf I-\mathbf X)^{-1}
\Big[
(\mathbf I\otimes \mathbf B')\,\mathbf U(x,x+1,x+2;t)
\Big] \\[4pt]
&\quad
+ q\,(\mathbf I-\mathbf Y)^{-1}
\Big[
(\mathbf B\otimes \mathbf I)\,\mathbf U(x,x+1,x+2;t)
\Big] \\[4pt]
&\quad
+ q\,(\mathbf I-\mathbf Y)^{-1}
\Big[
(\mathbf B'\otimes \mathbf I)(\mathbf I\otimes \mathbf B')\,
\mathbf U(x+1,x+2,x+3;t)
\Big] \\[4pt]
&\quad
- 3\,\mathbf U(x,x+1,x+2;t).
\end{aligned}
\end{equation}
This expression shows that all contributions to the evolution of admissible configurations are generated by compositions of two-particle interaction operators, thereby making the two-particle reducibility of the dynamics explicit.

The transition rates are encoded in the coefficient matrices appearing in \eqref{1149pm}. For example, the rate of transition from the configuration \((x-1,x,x+1,\nu)\) to \((x,x+1,x+2,\pi)\) is given by
\[
p\,\langle \pi \mid (\mathbf I-\mathbf X)^{-1}
(\mathbf I \otimes \mathbf B)(\mathbf B \otimes \mathbf I)\mid \nu \rangle.
\]
Using Lemma~\ref{100pm326} and Proposition~\ref{1259pm}, we obtain, for instance,
\begin{equation}
p\langle iii \mid
(\mathbf I-\mathbf X)^{-1}
(\mathbf I \otimes \mathbf B)(\mathbf B \otimes \mathbf I)
\mid iii \rangle
=
\frac{p\mu_i^2}{1-\mu_i\lambda_i}.
\end{equation}
This is the rate at which the leftmost particle of species \(i\) at position \(x-1\) jumps over two adjacent particles of the same species and moves to \(x+2\).

On the other hand, when \(i<j\), we have
\begin{equation}
p\langle iji \mid
(\mathbf I-\mathbf X)^{-1}
(\mathbf I \otimes \mathbf B)(\mathbf B \otimes \mathbf I)
\mid iij \rangle
=
p\mu_i.
\end{equation}
This corresponds to the rate at which the leftmost particle of species \(i\) at position \(x-1\) jumps over two particles in front of it—one of species \(i\) and one of species \(j\)—and moves to position \(x+2\).

These examples illustrate that the transition rates depend nontrivially on the local species configuration, and in particular on the number and arrangement of particles of the same species within the local block.

\subsection{General $n$: Two-particle reducibility}\label{340pm412}
Let $\mathbf{U}(x_1,\dots,x_n;t)$ be the $N^n\times N^n$ matrix-valued function whose
$(\pi,\nu)$-entry gives the transition probability from the initial species configuration $\nu$ and spacial configuration $Y$ (omitted in the notation)
to the final species configuration $\pi$ and spatial configurations satisfying
$x_1<\cdots<x_n$ at time $t$. We extend $\mathbf{U}$ to all $(x_1,\dots,x_n)\in\mathbb Z^n$
via the boundary conditions introduced below.

The function $\mathbf{U}(x_1,\dots,x_n;t)$ satisfies the free evolution equation,
corresponding to independent particle motion,
\begin{equation}\label{1055pm325-n}
\frac{d}{dt}\mathbf U(x_1,\dots,x_n;t)
=
\sum_{i=1}^n \bigl( p\mathcal{T}_i^- + q\mathcal{T}_i^+ - 1 \bigr)\mathbf U(x_1,\dots,x_n;t),
\qquad (x_1,\dots,x_n)\in \mathbb Z^n,
\end{equation}
where $\mathcal{T}_i^{\pm}$ denotes the operator that shifts the $i$-th spatial coordinate by $\pm1$.

The effect of interactions is incorporated through boundary conditions, which ensure that the induced evolution on the physical region $x_1<\cdots<x_n$ coincides with that of the original interacting particle system. The boundary conditions are naturally extended from \eqref{212am324} as follows: for each
$j=1,\dots,n-1$,
\begin{eqnarray}\label{bdn}
&&\mathbf U(x_1,\dots,x_{j-1},x,x,x_{j+2},\dots,x_n;t) \notag\\[4pt]
&=&
\bigl( \mathbf I^{\otimes (j-1)} \otimes \mathbf B \otimes \mathbf I^{\otimes (n-j-1)} \bigr)
\,\mathbf U(x_1,\dots,x_{j-1},x-1,x,x_{j+2},\dots,x_n;t)
\notag\\[6pt]
&&\qquad
+
\bigl( \mathbf I^{\otimes (j-1)} \otimes \mathbf B' \otimes \mathbf I^{\otimes (n-j-1)} \bigr)
\,\mathbf U(x_1,\dots,x_{j-1},x,x+1,x_{j+2},\dots,x_n;t),
\end{eqnarray}
for all $x,x_1,\dots,x_{j-1},x_{j+2},\dots,x_n\in\mathbb Z$, where $\mathbf{I}$ denotes the $N\times N$ identity matrix.

Thus, the free evolution equation \eqref{1055pm325-n}, together with the boundary conditions \eqref{bdn}, determines the dynamics entirely on the physical region
\[
\{(x_1,\dots,x_n)\in\mathbb Z^n:\ x_1<\cdots<x_n\}.
\]
We now explain how non-admissible configurations arising in the master equation,
due to the action of the operators $\mathcal{T}_i^{\pm}$ on adjacent particles,
are eliminated using the boundary condition \eqref{bdn}. These cases will be treated separately below.

\subsubsection{Non-admissible configurations arising from the operators $\mathcal{T}_i^-$}\label{444am}
The operators $\mathcal{T}_i^-$ produce non-admissible configurations in which a consecutive block
terminates with a repeated pair, of the form
\begin{equation}\label{maxblock2}
(\,x_1,\dots,x_{j-1},x,x+1,\dots,\underbrace{x+m-1,x+m-1}_{\text{repeated pair}},x_{j+m+1},\dots,x_n\,),
\end{equation}
for some $m=1,\dots,n-1$. Such terms arising in \eqref{1055pm325-n} can be eliminated recursively
using the boundary conditions. By repeated application of the boundary condition \eqref{bdn}, a configuration of the form \eqref{maxblock2}
can be transformed into configurations of the form
\begin{align}\label{7090m330}
(&x_1,\dots,x_{j-1},x,\dots,x+i-2,\underbrace{x+i-1,x+i-1}_{\text{repeated pair}},x+i,\dots,x+m-1,x_{j+m+1},\dots,x_n;\,t),
\end{align}
for some $i \in \{1,\dots, m\}$. For each $i=1,\dots,m$, define
\begin{align*}
\mathbf{W}_{m,i}
:=
\mathbf U(&x_1,\dots,x_{j-1},x,\dots,x+i-2,\\
&x+i-1,x+i-1,x+i,\dots,x+m-1,x_{j+m+1},\dots,x_n;\,t),
\end{align*}
where the repeated pair $(x+i-1,x+i-1)$ occupies the $i$-th and $(i+1)$-th positions within the consecutive block.

We also define
\begin{align*}
\mathbf{W}_{m,0}
&:=
\mathbf U(x_1,\dots,x_{j-1},x-1,x,\dots,x+m-1,x_{j+m+1},\dots,x_n;\,t),\\
\mathbf{W}_{m,m+1}
&:=
\mathbf U(x_1,\dots,x_{j-1},x,x+1,\dots,x+m,x_{j+m+1},\dots,x_n;\,t),
\end{align*}
which correspond to admissible configurations without repeated pairs in the block.
\begin{lemma}
For $i=1,\dots,m$,
\begin{equation}
\mathbf{W}_{m,i}
=
\mathbf{\mathcal B}_i\,\mathbf{W}_{m,i-1}
+
\mathbf{\mathcal B}_i'\,\mathbf{W}_{m,i+1},
\label{eq:block-recurrence}
\end{equation}
where
\[
\mathbf{\mathcal B}_i
=
\mathbf I^{\otimes (j+i-2)}\otimes \mathbf B \otimes \mathbf I^{\otimes (n-j-i)},
\qquad
\mathbf{\mathcal B}_i'
=
\mathbf I^{\otimes (j+i-2)}\otimes \mathbf B' \otimes \mathbf I^{\otimes (n-j-i)},
\]
and $\mathbf{I}$ denotes the $N \times N$ identity matrix.
\end{lemma}
\begin{proof}
This follows directly from the boundary condition \eqref{bdn}, applied to the repeated pair
at positions $(j+i-1,j+i)$.
\end{proof}
Equation \eqref{eq:block-recurrence} defines a linear system for $\mathbf{W}_{m,i}$, $i=1,\dots,m$,
which can be solved recursively. By eliminating the intermediate vectors
$\mathbf{W}_{m,1},\dots,\mathbf{W}_{m,m}$, each $\mathbf{W}_{m,i}$ can be expressed in terms of
$\mathbf{W}_{m,0}$ and $\mathbf{W}_{m,m+1}$ as
\begin{equation}
\mathbf{W}_{m,i}
=
\mathbf{\mathcal L}_i\,\mathbf{W}_{m,0}
+
\mathbf{\mathcal L}_i'\,\mathbf{W}_{m,m+1},
\label{eq:block-recurrence1}
\end{equation}
for suitable matrices $\mathbf{\mathcal L}_i$ and $\mathbf{\mathcal L}_i'$, provided that certain invertibility conditions hold.

To illustrate the procedure, consider the case $m=2$, for which there are two non-admissible terms $\mathbf{W}_{2,1}$ and $\mathbf{W}_{2,2}$. Then
\[
\mathbf{W}_{2,2}
=
\mathbf{\mathcal B}_2\,\mathbf{W}_{2,1}
+
\mathbf{\mathcal B}_2'\,\mathbf{W}_{2,3}
=
\mathbf{\mathcal B}_2\bigl(
\mathbf{\mathcal B}_1\,\mathbf{W}_{2,0}
+
\mathbf{\mathcal B}_1'\,\mathbf{W}_{2,2}
\bigr)
+
\mathbf{\mathcal B}_2'\,\mathbf{W}_{2,3},
\]
which yields
\[
\bigl(\mathbf I - \mathbf{\mathcal B}_2 \mathbf{\mathcal B}_1'\bigr)\mathbf{W}_{2,2}
=
\mathbf{\mathcal B}_2 \mathbf{\mathcal B}_1\,\mathbf{W}_{2,0}
+
\mathbf{\mathcal B}_2'\,\mathbf{W}_{2,3}.
\]
Assuming that $\mathbf I - \mathbf{\mathcal B}_2 \mathbf{\mathcal B}_1'$ is invertible, we obtain
\[
\mathbf{W}_{2,2}
=
\bigl(\mathbf I - \mathbf{\mathcal B}_2 \mathbf{\mathcal B}_1'\bigr)^{-1}
\mathbf{\mathcal B}_2 \mathbf{\mathcal B}_1\,\mathbf{W}_{2,0}
+
\bigl(\mathbf I - \mathbf{\mathcal B}_2 \mathbf{\mathcal B}_1'\bigr)^{-1}
\mathbf{\mathcal B}_2'\,\mathbf{W}_{2,3}.
\]
Substituting this into
\[
\mathbf{W}_{2,1}
=
\mathbf{\mathcal B}_1\,\mathbf{W}_{2,0}
+
\mathbf{\mathcal B}_1'\,\mathbf{W}_{2,2},
\]
we obtain expressions for both $\mathbf{W}_{2,1}$ and $\mathbf{W}_{2,2}$ in terms of
$\mathbf{W}_{2,0}$ and $\mathbf{W}_{2,3}$. This coincides with the case analyzed in Section~\ref{302am1} and illustrates the general elimination mechanism.

For general $m$, iterating this procedure leads to the recursion
\begin{equation}\label{1122am331}
\mathfrak A_0=\mathbf{I},\qquad
\mathfrak A_k=\mathbf{I}-\mathbf{\mathcal B}_{k+1}\mathfrak A_{k-1}^{-1}\mathbf{\mathcal B}'_k,
\qquad (k\ge1),
\end{equation}
provided that the operators $\mathfrak A_1,\dots,\mathfrak A_{k-1}$ are invertible.
Under this condition, the reduction \eqref{eq:block-recurrence1} holds.
\begin{remark}
In the homogeneous case $\mu_i = 1$ for all $i$, the invertibility of $\mathfrak A_k$ for all $k$ was established in \cite{Lee2026}, together with an explicit formula for its inverse.

In the general case, proving invertibility of the $N^n \times N^n$ matrix $\mathfrak A_k$ appears to be highly nontrivial. Although $\mathfrak A_k$ becomes block-diagonal after reordering the basis according to species compositions, establishing invertibility of each block remains difficult.

Numerical evidence  supports invertibility in general. In particular, for small systems (up to $N=n=4$), we computed the spectral radius of
\[
\mathbf{\mathcal B}_{k+1}\mathfrak A_{k-1}^{-1}\mathbf{\mathcal B}'_k
\]
for all $k$, and found it to be strictly less than $1$. This implies invertibility of $\mathfrak A_k$ in these cases.

For general $N$ and $n$, invertibility can be proved in the following settings:
(i) the binary case $\mu_i \in \{0,1\}$ for arbitrary special compositions, and
(ii) the continuous case $\mu_i \in (0,1)$ for certain special species compositions.

These results lead us to conjecture that $\mathfrak A_k$ is invertible for all parameters $\mu_i \in [0,1]$.
\end{remark}

\paragraph{Invertibility of $\mathfrak A_k$ for the binary-parameter setting $\mu_i \in \{0,1\}$}
In the  case where each parameter satisfies $\mu_i \in \{0,1\}$, one has
$\mu_i \lambda_i = 0$ for all $i$, since $\lambda_i = 1 - \mu_i$.
The following result extends Corollary~3.9 of \cite{Lee2026} to this setting.
\begin{corollary}\label{1126am331}
Suppose that $\mu_i \in \{0,1\}$ for each $i$. Then $\mathfrak A_k$ defined in \eqref{1122am331}
is invertible for all $k \ge 1$, and
\[
\mathfrak A_k^{-1}
=
\mathbf{I}
+
\mathbf{\mathcal B}_{k+1}\mathfrak A_{k-1}^{-1}\mathbf{\mathcal B}'_k.
\]
\end{corollary}
\begin{proof}
Let
\[
\mathcal{X}_k:=\mathbf{\mathcal B}_{k+1}\mathfrak A_{k-1}^{-1}\mathbf{\mathcal B}_k'.
\]
Then $\mathfrak A_k=\mathbf I - \mathcal{X}_k.$ It therefore suffices to show that \(\mathcal{X}_k^2=\mathbf{0}\). In the binary-parameter setting \(\mu_i\in\{0,1\}\), we have \(\mu_i\lambda_i=0\) for all \(i\).
As a consequence, the local nilpotency property established in \cite{Lee2026}
for the operators \(\mathbf{\mathcal B}_i\) and \(\mathbf{\mathcal B}_i'\) remains valid.
In particular, the key relation used in \cite{Lee2026},
\[
(\mathbf{\mathcal B}_{k+1}\cdots \mathbf{\mathcal B}_{2}\mathbf{\mathcal B}_{1}'\cdots \mathbf{\mathcal B}_k')^2 = \mathbf{0},
\]
continues to hold. The same inductive argument as in Corollary~3.9 of \cite{Lee2026} therefore applies,
and yields
\[
\mathcal{X}_k^2
=
\bigl(\mathbf{\mathcal B}_{k+1}\mathfrak A_{k-1}^{-1}\mathbf{\mathcal B}_k'\bigr)^2
=
\mathbf{0}.
\]
Hence,
\begin{equation*}
\mathfrak A_k^{-1}
=
\mathbf I + \mathcal{X}_k
=
\mathbf I + \mathbf{\mathcal B}_{k+1}\mathfrak A_{k-1}^{-1}\mathbf{\mathcal B}_k'.
\end{equation*}
\end{proof}
\paragraph{Invertibility of $\mathfrak A_k$ for the continuous-parameter setting $\mu_i\in (0,1)$ on certain invariant subspaces.}\label{412ama}
Let $M=[\nu_1,\dots,\nu_n]$ be a multiset describing the species of $n$ particles, and let $V_M$ be the vector space spanned by all permutations of $M$. We establish invertibility of $\mathfrak A_k$ on the invariant subspaces corresponding to the following multisets:
\begin{itemize}
\item[(a)] $[i,\dots,i]$,
\item[(b)] $[i,j,\dots,j]$ with $i\neq j$, and
\item[(c)] $[\nu_1,\dots,\nu_n]$ with all $\nu_\ell$ distinct.
\end{itemize}
Since both $\mathcal B_r$ and $\mathcal B_r'$ preserve the multiset of species, the subspace $V_M$ is invariant under $\mathfrak A_k$ whenever $\mathfrak A_{k-1}$ is well defined.
\begin{itemize}
\item[(a)] (Invertibility on $V_{[i,\dots,i]}$) Since $V_{[i,\dots,i]}$ is one-dimensional, the operator $\mathfrak A_k$ acts as a scalar:
\[
\mathfrak A_k |i\cdots i\rangle = a_{k,i}\,|i\cdots i\rangle.
\]
Thus,
\begin{equation}\label{1231am}
a_{0,i}=1,
\qquad
a_{k,i}=1-\frac{\mu_i\lambda_i}{a_{k-1,i}},
\quad (k\ge1),
\end{equation}
whenever $a_{k-1,i}\neq 0$. Set $\alpha_i:=\mu_i\lambda_i$, so that $0\le \alpha_i\le \frac{1}{4}$, and let
\[
r_i:=\frac{1+\sqrt{1-4\alpha_i}}{2}.
\]
Then $r_i\in[1/2,1]$ is the larger root of $r=1-\alpha_i/r$. A simple induction shows that
\[
a_{k,i} \ge r_i > 0
\qquad \text{for all } k\ge 0,
\]
and hence the recursion is well defined at every step. Therefore, $\mathfrak A_k$ is invertible on $V_{[i,\dots,i]}$.
\item [(b)] (Invertibility of $\mathfrak A_k$ on $V_{[i,j,\dots,j]}$ where $i \neq j$)  Recall that throughout the argument below, $\mathbf{I}$ denotes the identity matrix of the appropriate dimension in the context. Let
\begin{equation}\label{414am}
\mathcal X_k := \mathbf{\mathcal B}_{k+1}\mathfrak A_{k-1}^{-1}\mathbf{\mathcal B}'_k,
\end{equation}
so that $\mathfrak A_k = \mathbf{I} - \mathcal X_k$. Since the operators in \eqref{414am} preserve the multiset of species, they decompose into blocks corresponding to species compositions. We restrict to the block corresponding to $V_{[i,j,\dots,j]}$, which is $n$-dimensional.

We show that
\[
\rho(\mathcal X_k) < 1 \quad \text{for all } k,
\]
which implies that $\mathfrak A_k$ is invertible. Without loss of generality, assume $i=2$ and $j=1$. For $r=0,\dots,n-1$, define
\[
e_r := |1^r 2 1^{n-r-1}\rangle,
\]
so that $\{e_0,\dots,e_{n-1}\}$ forms a basis of $V_{[2,1,\dots,1]}$. We begin by showing that $\mathfrak A_1$ and $\mathfrak A_2$ are invertible, and then extend this argument to the general case.  

The operators $\mathbf{\mathcal B}_i$ and $\mathbf{\mathcal B}'_i$ act locally by swapping adjacent particles with weights determined by $\mathbf B$ and $\mathbf B'$. A direct computation shows that
\[
\mathbf{\mathcal B}_1' e_0 = e_1,\quad
\mathbf{\mathcal B}_1' e_1 = 0,\quad
\mathbf{\mathcal B}_1' e_r = \lambda_1 e_r \quad (r\ge 2),
\]
and hence $\mathcal X_1 = \mathbf{\mathcal B}_2 \mathbf{\mathcal B}_1'$
acts as
\[
\mathcal X_1 e_0 = 0,\quad
\mathcal X_1 e_1 = 0,\quad
\mathcal X_1 e_2 = \lambda_1 e_1,\quad
\mathcal X_1 e_r = \mu_1\lambda_1 e_r \quad (r\ge 3).
\]

Thus, $\mathcal X_1$ has a block upper-triangular form with a nilpotent  block and a diagonal block with the diagonal entry $\mu_1\lambda_1$. In particular,
\[
\operatorname{Spec}(\mathcal X_1) \subset \{0,\,\mu_1\lambda_1\},
\qquad
\rho(\mathcal X_1) = \mu_1\lambda_1 < 1.
\]
Hence $\mathfrak A_1 = \mathbf I - \mathcal X_1$ is invertible.

Expanding the inverse of $\mathfrak A_1$, we obtain
\begin{equation}\label{1116pm}
\mathcal X_2
=
\mathbf{\mathcal B}_{3}(\mathbf{I} - \mathbf{\mathcal B}_{2}\mathbf{\mathcal B}'_{1})^{-1} \mathbf{\mathcal B}'_2
=
\mathbf{\mathcal B}_{3}\sum_{m\ge0}(\mathbf{\mathcal B}_{2}\mathbf{\mathcal B}'_{1})^m \mathbf{\mathcal B}'_2.
\end{equation}
Acting term-by-term on the basis vectors, a direct computation shows that $\mathcal X_2$ again has a block upper-triangular structure, with a nilpotent block and  a diagonal block with diagonal entries
\[
\mu_1\lambda_1
\quad \text{and} \quad
\frac{\mu_1\lambda_1}{1-\mu_1\lambda_1}.
\]
Since $\mu_1\lambda_1 \le \tfrac14$, all eigenvalues are strictly less than $1$, and hence
\[
\rho(\mathcal X_2) < 1.
\]
Therefore, $\mathfrak A_2 = \mathbf I - \mathcal X_2$ is invertible.

These observations extend to general $k$, as formalized in the following lemma.

\begin{lemma}\label{lem:block-form-Xk}
For each $k\ge 1$, the matrix of $\mathcal X_k$ with respect to the ordered basis
\[
(e_0,\dots,e_{k-1}\mid e_k,e_{k+1}\mid e_{k+2},\dots,e_{n-1})
\]
has the block diagonal  form
\[
\mathcal X_k=
\left(
\begin{array}{c|cc|c}
D_k^- & 0 & 0 & 0\\
\hline
0 & 0 & * & 0\\
0 & 0 & 0 & 0\\
\hline
0 & 0 & 0 & D_{n-k-2}^+
\end{array}
\right),
\]
where $D_k^-$ and $D_{n-k-2}^+$ are diagonal matrices of dimensions $k$ and $n-k-2$, respectively. In particular, the middle $2\times2$ block is nilpotent.
\end{lemma}
\begin{proof}
The proof is given in Appendix~\ref{412129am}.
\end{proof}

\begin{lemma}\label{1203ama}
For \(1\le k\le n-2\), let
\[
U_k^-:=\operatorname{span}\{e_0,\dots,e_k\}.
\]
Then \(\mathcal X_{k+1}\) acts on \(U_k^-\) as a scalar:
\[
\mathcal X_{k+1}|_{U_k^-}=c_k \mathbf{I},
\]
where \(c_k\) is the scalar such that
\[
\mathcal X_k|1\cdots1\rangle=c_k|1\cdots1\rangle.
\]
\end{lemma}
\begin{proof}
The proof is given in Appendix~\ref{412129am}.
\end{proof}

\begin{proposition}\label{prop:sector-invertibility-mixed}
For each \(1\le k\le n-2\), the operator \(\mathfrak A_k\) restricted to \(V_{[2,1,\dots,1]}\) is invertible.
\end{proposition}

\begin{proof}
By Lemma~\ref{lem:block-form-Xk}, with respect to the ordered basis
\[
(e_0,\dots,e_{k-1}\mid e_k,e_{k+1}\mid e_{k+2},\dots,e_{n-1}),
\]
the operator \(\mathfrak A_k=\mathbf I-\mathcal X_k\) has the block form
\[
\mathfrak A_k
=
\begin{pmatrix}
\mathbf I - D_k^- & 0 & 0 & 0\\
0 & 1 & * & 0\\
0 & 0 & 1 & 0\\
0 & 0 & 0 & \mathbf I - D_{n-k-2}^+
\end{pmatrix}.
\]
Therefore, it suffices to show that all diagonal entries of \(D_k^-\) and \(D_{n-k-2}^+\) lie in \([0,1)\).

\medskip
\noindent
\textit{Left block.}
By Lemma~\ref{1203ama}, on
\[
U_{k-1}^-=\operatorname{span}\{e_0,\dots,e_{k-1}\},
\]
we have
\[
\mathcal X_k\big|_{U_{k-1}^-}=c_{k-1}\mathbf I,
\qquad
c_{k-1}=\frac{\alpha}{a_{k-2,1}},
\qquad
\alpha:=\mu_1\lambda_1.
\]
Moreover,
\[
a_{k-1,1}=1-\frac{\alpha}{a_{k-2,1}}=1-c_{k-1}.
\]
Since part (a) shows that \(a_{k-1,1}>0\), it follows that
\[
0\le c_{k-1}<1.
\]
Hence all diagonal entries of \(D_k^-\) lie in \([0,1)\).

\medskip
\noindent
\textit{Right block.}
Consider the subspace
\[
U_k^+=\operatorname{span}\{e_{k+2},\dots,e_{n-1}\}.
\]
For vectors in \(U_k^+\), the particle of species \(2\) lies strictly to the right of the sites on which \(\mathbf{\mathcal B}_k'\) and \(\mathbf{\mathcal B}_{k+1}\) act. Therefore, on \(U_k^+\), the operator \(\mathcal X_k\) coincides with the corresponding one-species operator on the remaining coordinates. In particular, its eigenvalues are of the form
\[
\frac{\alpha}{a_{m,1}}
\]
for suitable \(m\ge0\). Since part (a) implies \(a_{j,1}>0\) for all $j$ and
\[
\frac{\alpha}{a_{m,1}}=1-a_{m+1,1}<1,
\]
all diagonal entries of \(D_{n-k-2}^+\) also lie in \([0,1)\).

\medskip
\noindent
\textit{Middle block.}
The middle \(2\times2\) block of \(\mathfrak A_k\) is upper triangular with diagonal entries equal to \(1\), and is therefore invertible.

\medskip
\noindent
Thus each diagonal block of \(\mathfrak A_k\) is invertible, and hence \(\mathfrak A_k\) is invertible on \(V_{[2,1,\dots,1]}\).
\end{proof}
\item[(c)] (Invertibility of \(\mathfrak A_k\) on \(V_{[\nu_1,\dots,\nu_n]}\) with all \(\nu_i\) distinct.)
This class of multisets was already treated in \cite{Lee2026}, where the invertibility of the corresponding operators was established. Since no same-species interactions occur in this setting, the argument of \cite{Lee2026} applies without modification. Therefore, \(\mathfrak A_k\) is invertible on \(V_{[\nu_1,\dots,\nu_n]}\).
\end{itemize}
\subsubsection{Non-admissible configurations arising from the operators $\mathcal{T}_i^+$}
The treatment of the operators $\mathcal{T}_i^+$ is closely parallel to that of $\mathcal{T}_i^-$,
but involves the operator
\[
\mathbf{Y} := (\mathbf{B}' \otimes \mathbf{I})(\mathbf{I} \otimes \mathbf{B}),
\]
instead of $\mathbf{X} = (\mathbf{I} \otimes \mathbf{B})(\mathbf{B}' \otimes \mathbf{I})$, introduced in Section~\ref{302am1}.

In Sections~\ref{302am1} and~\ref{302am2}, we showed how the non-admissible configurations arising from $\mathcal{T}_i^+$ are eliminated via the operator $\mathbf{Y}$ in the case $n=3$. The same elimination procedure extends to the general $n$ case, with the necessary modifications of the arguments from the previous sections (for $\mathcal{T}_i^-$). Therefore, We omit the details, as they follow by a straightforward adaptation of the arguments for $\mathcal{T}_i^-$.
\subsection{Effective swap rates for consecutive particles}\label{342pm412}
In  the Introduction, the dynamics of the model were defined through two-particle interactions. The full evolution is determined by these local rules together with the boundary reduction procedure.

The goal of this subsection is to describe the effective swap rates for configurations in which particles occupy consecutive sites. The key observation is that when a particle traverses a block of consecutive particles, the resulting swap rate depends only on the number of particles of the \emph{same species} encountered along the path. Equivalently, the effective rate coincides with that of a corresponding single-species system. This phenomenon already appeared in the three-particle analysis in Section~\ref{dynamics3}.

To make this precise, consider the transition
\begin{equation}\label{133am}
(x-1,x,\dots,x+n-2;\nu_1\cdots \nu_n)
\;\longrightarrow\;
(x,x+1,\dots,x+n-1;\nu_2\cdots \nu_n\nu_1),
\end{equation}
under the assumption that
\[
\nu_1 \le \nu_k \qquad \text{for all } k \ge 2.
\]
That is, the particle of species \(\nu_1\) at site \(x-1\) moves across the entire block and reaches the site \(x+n-1\), initially occupied by a hole (viewed as species \(0\)).

During this process, the particle interacts successively with particles to its right. These interactions can be described via intermediate hidden states arising from the two-particle dynamics (see \cite{Lee2026}). When two particles have different species, the local swap occurs deterministically, contributing no nontrivial factor to the rate. Consequently, only interactions with particles of the same species affect the effective rate.

This is illustrated in Figure~\ref{fig2}. In each of the three cases, the moving particle (of species $2$) encounters three particles of the same species, although the surrounding particles differ. Since interactions with other species do not affect the rate, all three configurations yield the same effective swap rate.

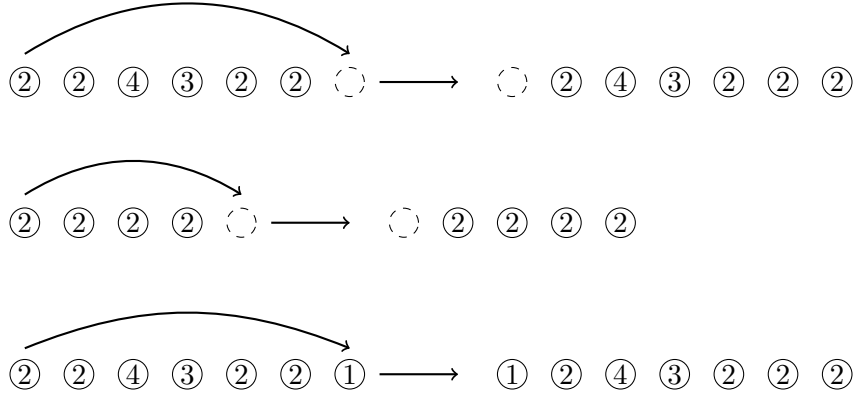
\begin{figure}[h]
\centering
\begin{tikzpicture}[scale=0.72]

\tikzset{
  circ/.style={draw, circle, inner sep=0.8pt, minimum size=11pt}
}


\node[circ] at (1,1.2) {2};
\node[circ] at (2,1.2) {2};
\node[circ] at (3,1.2) {4};
\node[circ] at (4,1.2) {3};
\node[circ] at (5,1.2) {2};
\node[circ] at (6,1.2) {2};
\node[circ,dashed] at (7,1.2) {};

\draw[->, thick] (1,1.72) to[out=32,in=148] (7,1.72);
\draw[->, thick] (7.55,1.2) -- (9,1.2);

\node[circ,dashed] at (10,1.2) {};
\node[circ] at (11,1.2) {2};
\node[circ] at (12,1.2) {4};
\node[circ] at (13,1.2) {3};
\node[circ] at (14,1.2) {2};
\node[circ] at (15,1.2) {2};
\node[circ] at (16,1.2) {2};


\node[circ] at (1,-1.4) {2};
\node[circ] at (2,-1.4) {2};
\node[circ] at (3,-1.4) {2};
\node[circ] at (4,-1.4) {2};
\node[circ,dashed] at (5,-1.4) {};

\draw[->, thick] (1,-0.88) to[out=32,in=148] (5,-0.88);
\draw[->, thick] (5.55,-1.4) -- (7,-1.4);

\node[circ,dashed] at (8,-1.4) {};
\node[circ] at (9,-1.4) {2};
\node[circ] at (10,-1.4) {2};
\node[circ] at (11,-1.4) {2};
\node[circ] at (12,-1.4) {2};


\node[circ] at (1,-4.2) {2};
\node[circ] at (2,-4.2) {2};
\node[circ] at (3,-4.2) {4};
\node[circ] at (4,-4.2) {3};
\node[circ] at (5,-4.2) {2};
\node[circ] at (6,-4.2) {2};
\node[circ] at (7,-4.2) {1};

\draw[->, thick] (1,-3.72) to[out=22,in=158] (7,-3.72);
\draw[->, thick] (7.55,-4.2) -- (9,-4.2);

\node[circ] at (10,-4.2) {1};
\node[circ] at (11,-4.2) {2};
\node[circ] at (12,-4.2) {4};
\node[circ] at (13,-4.2) {3};
\node[circ] at (14,-4.2) {2};
\node[circ] at (15,-4.2) {2};
\node[circ] at (16,-4.2) {2};

\end{tikzpicture}
\caption{The transition rates in all three cases are the same.}
\label{fig2}
\end{figure}
Let us determine the rate of such a transition in Figure~\ref{fig2}. Without loss of generality, consider
\begin{equation}\label{3452}
(x-1,x,\dots,x+n-2; i\cdots i)
\;\longrightarrow\;
(x,x+1,\dots,x+n-1; i\cdots i).
\end{equation}

By the recursion procedure in Section~\ref{444am}, the coefficient matrix of $\mathbf{U}(x-1,x,\dots,x+n-2;t)$
in the master equation for $\mathbf{U}(x,x+1,\dots,x+n-1;t)$
is
\begin{equation}\label{eq:coeff-same-species}
\mathfrak A_{n-2}^{-1}\mathcal B_{n-1}\mathfrak A_{n-3}^{-1}\mathcal B_{n-2}
\cdots
\mathfrak A_1^{-1}\mathcal B_2\mathcal B_1,
\end{equation}
obtained by solving \eqref{eq:block-recurrence} for $\mathbf W_{n-1,n-1}$ and extracting the coefficient
$\mathcal L_{n-1}$ in \eqref{eq:block-recurrence1} (see Lemma~\ref{513ama} in Appendix~\ref{514ama}). Hence
\[
p\bigl\langle i\cdots i \bigm|
\mathfrak A_{n-2}^{-1}\mathcal B_{n-1}\mathfrak A_{n-3}^{-1}\mathcal B_{n-2}\cdots
\mathfrak A_1^{-1}\mathcal B_2\mathcal B_1
\bigm| i\cdots i \bigr\rangle
\]
gives the rate of the transition \eqref{3452}.

\begin{proposition}
The rate of the transition in \eqref{3452} is
\begin{equation}\label{eq:AKK-recovered}
r_{n-1}^{(i)}
=
p\,\frac{\mu_i^{\,n-1}}{\mu_i^{\,n-1}+\mu_i^{\,n-2}\lambda_i+\cdots+\lambda_i^{\,n-1}}=
\frac{p}{1+\lambda_i/\mu_i+\cdots+(\lambda_i/\mu_i)^{n-1}}.
\end{equation}
\end{proposition}

\begin{proof}
On the sector $|i\cdots i\rangle$, each $\mathcal B_j$ acts by $\mu_i$ and each $\mathcal B_j'$ acts by $\lambda_i$. Hence
\[
\bigl\langle i\cdots i \bigm|
\mathfrak A_{n-2}^{-1}\mathcal B_{n-1}\cdots \mathfrak A_1^{-1}\mathcal B_2\mathcal B_1
\bigm| i\cdots i \bigr\rangle
=
\frac{\mu_i^{\,n-1}}{a_{1,i}\cdots a_{n-2,i}},
\]
where
\[
a_{0,i}=1,
\qquad
a_{k,i}=1-\frac{\mu_i\lambda_i}{a_{k-1,i}}.
\]
To evaluate the product, we introduce the ratio ansatz
\[
a_{k,i}=\frac{S_{k+1}^{(i)}}{S_k^{(i)}}.
\]
Substituting into the recursion yields the linear recurrence
\[
S_{k+1}^{(i)}=S_k^{(i)}-\mu_i\lambda_i S_{k-1}^{(i)}.
\]
This recurrence has characteristic equation $r^2-r+\mu_i\lambda_i=0$, with roots $\mu_i$ and $\lambda_i$. Hence
\[
S_k^{(i)}=C_1 \mu_i^{\,k}+C_2 \lambda_i^{\,k}.
\]
Imposing $S_0^{(i)}=1$ and $S_1^{(i)}=\mu_i+\lambda_i=1$ determines the constants and gives
\[
S_k^{(i)}=\frac{\mu_i^{k+1}-\lambda_i^{k+1}}{\mu_i-\lambda_i}=\sum_{r=0}^k \mu^{k-r}\lambda^r.
\]
Since $S_1^{(i)}=\mu_i+\lambda_i=1$, the product telescopes:
\[
a_{1,i}\cdots a_{n-2,i}
=
\frac{S_{n-1}^{(i)}}{S_1^{(i)}}
=
S_{n-1}^{(i)}.
\]
Therefore
\[
r_{n-1}^{(i)}
=
p\,\frac{\mu_i^{\,n-1}}{S_{n-1}^{(i)}}
=
p\,\frac{\mu_i^{\,n-1}}{\mu_i^{\,n-1}+\mu_i^{\,n-2}\lambda_i+\cdots+\lambda_i^{\,n-1}}.
\]
\end{proof}
\begin{remark}
Our result \eqref{eq:AKK-recovered} provides an alternative derivation of equation (35) in \cite{Ali}.
\end{remark}
\subsection{Yang--Baxter integrability}\label{343pm412}
We have established two-particle reducibility for species-dependent binary parameters $\mu_i \in \{0,1\}$ for arbitrary species compositions, and for continuous parameters $\mu_i \in (0,1)$ in several nontrivial cases.

The next step in establishing integrability is to verify that the associated scattering matrix satisfies the Yang--Baxter equation. The $(i,j)$ sector with $i<j$ of the scattering matrix for the present model, obtained following the construction in \cite{Lee-2020,Lee2026}, is given by
\begin{equation}\label{Rmatirx}
\mathbf{R}_{\beta\alpha}:=\kbordermatrix{
      & ii & ij & ji & jj \\
ii & -\dfrac{(\mu_i + \lambda_i \xi_{\alpha}\xi_{\beta}-\xi_{\alpha})\xi_{\beta}}
          {(\mu_i + \lambda_i \xi_{\alpha}\xi_{\beta}-\xi_{\beta})\xi_{\alpha}}
    & 0 & 0 & 0 \\[6pt]
ij & 0 & 0 & \xi_{\beta} & 0 \\[6pt]
ji & 0 & \dfrac{1}{\xi_{\alpha}} & 0 & 0 \\[6pt]
jj & 0 & 0 & 0 &
    -\dfrac{(\mu_j + \lambda_j \xi_{\alpha}\xi_{\beta}-\xi_{\alpha})\xi_{\beta}}
          {(\mu_j + \lambda_j \xi_{\alpha}\xi_{\beta}-\xi_{\beta})\xi_{\alpha}}
}.
\end{equation}

Since the Yang--Baxter equation acts nontrivially only on three-particle sectors, it suffices to consider a three-particle system with at most three species. Accordingly, we extend \eqref{Rmatirx} to a $9\times 9$ matrix whose rows and columns are indexed by $11,12,\dots,33$. Even when the parameters $\mu_1,\mu_2,\mu_3$ are distinct, the Yang--Baxter equation remains satisfied for all $\mu_i \in [0,1]$.

\begin{theorem}
Let $\mathbf{R}_{\beta\alpha}$ be the $9\times 9$ scattering matrix defined by \eqref{Rmatirx}. Then, for distinct spectral parameters $\xi_{\alpha},\xi_{\beta},\xi_{\gamma}$, the Yang--Baxter equation
\begin{equation}\label{505pm}
(\mathbf{R}_{\gamma\beta}\otimes \mathbf{I})
(\mathbf{I}\otimes \mathbf{R}_{\gamma\alpha})
(\mathbf{R}_{\beta\alpha}\otimes \mathbf{I})
=
(\mathbf{I}\otimes \mathbf{R}_{\beta\alpha})
(\mathbf{R}_{\gamma\alpha}\otimes \mathbf{I})
(\mathbf{I}\otimes \mathbf{R}_{\gamma\beta})
\end{equation}
holds, where $\mathbf{I}$ denotes the $3\times 3$ identity matrix.
\end{theorem}

\begin{proof}
This identity can be verified by  direct computation of the corresponding $27\times 27$ matrices. Alternatively, following Lemma~A.2 of \cite{Lee-2024}, one observes that, after a suitable reordering of rows and columns, the matrices become block diagonal, with each block acting on an invariant subspace spanned by permutations of a fixed multiset of species. Therefore, it suffices to verify \eqref{505pm} on the blocks corresponding to the multisets $[i,i,i]$, $[i,i,j]$ (with $i\neq j$), and $[i,j,k]$ with $i,j,k$ all distinct.
\end{proof}
\section{Discussion}\label{Section4}

In the infinite-volume setting, additional subtleties arise in defining the model when infinitely many particles have nontrivial interaction parameters (i.e., $\mu_i \neq 0$). In the binary regime $\mu_i \in \{0,1\}$, for which we established integrability, the model admits a well-defined infinite-volume formulation. For instance, one may consider configurations in which infinitely many particles of a fixed species $i_0$ with $\mu_{i_0}=0$ (so that they behave as in TASEP among themselves) occupy sites $x \le -x_0$, while finitely many particles of other species, with parameters $\mu_i \in \{0,1\}$, occupy the remaining sites. Such configurations may be viewed as finite perturbations of a homogeneous TASEP background with multispecies long-range swap dynamics.

A central direction for future research in this setting is the analysis of asymptotic properties of observables. In particular, it is natural to investigate particle currents and the behavior of tagged particles, and to examine whether KPZ-type fluctuation behavior emerges in this model.

Unlike existing multispecies models, which have been extensively studied, it remains unclear whether the present model admits a coupling method suitable for asymptotic analysis. This suggests that alternative approaches may be required. A promising direction is to adapt the approach of Tracy and Widom for ASEP~\cite{TracyWidom2009} to the present model. In particular, it would be of interest to derive explicit formulas for relevant observables and to carry out their asymptotic analysis.

\appendix
\section{Additional Lemmas}\label{514ama}
\begin{lemma}\label{100pm326}
Let
\[
\mathbf X := (\mathbf I\otimes \mathbf B)(\mathbf B'\otimes \mathbf I),~~\mathbf X_0 := (\mathbf{I} \otimes \mathbf{B})(\mathbf{B} \otimes \mathbf{I}).
\]
Then,
\begin{equation*}
\mathbf X\mathbf X_0 = \sum_{i=1}^N  \mu_i^3\lambda_i \mathbf{E}_i ~~\textrm{where} ~~\mathbf{E}_i = |iii\rangle\langle iii|.
\end{equation*}
\end{lemma}
\begin{proof}
We  compute the action of the operator on the standard basis vectors $|\nu_1\nu_2\nu_3\rangle $
of \(\mathbb C^{N^3}\). Since
\(\mathbf B\otimes \mathbf I\) acts on the first two coordinates, we have
\[
(\mathbf B\otimes \mathbf I)|\nu_1\nu_2\nu_3\rangle
=
\begin{cases}
\mu_{\nu_1}\,|\nu_1\nu_1\nu_3\rangle, & \text{if } \nu_1=\nu_2,\\[4pt]
|\nu_2\nu_1\nu_3\rangle, & \text{if } \nu_1<\nu_2,\\[4pt]
0, & \text{if } \nu_1>\nu_2.
\end{cases}
\]
Applying \(\mathbf I\otimes \mathbf B\), which acts on the last two coordinates, yields
\[
\mathbf X_0|\nu_1\nu_2\nu_3\rangle
=
\begin{cases}
\mu_{\nu_1}^2\,|\nu_1\nu_1\nu_1\rangle, & \text{if } \nu_1=\nu_2=\nu_3,\\[4pt]
\mu_{\nu_1}\,|\nu_1\nu_3\nu_1\rangle, & \text{if } \nu_1=\nu_2<\nu_3,\\[4pt]
\mu_{\nu_2}\,|\nu_2\nu_2\nu_3\rangle, & \text{if } \nu_1<\nu_2=\nu_3,\\[4pt]
|\nu_2\nu_3\nu_1\rangle, & \text{if } \nu_1<\nu_2<\nu_3,\\[4pt]
0, & \text{otherwise}.
\end{cases}
\]
We now apply $\mathbf X=(\mathbf I\otimes \mathbf B)(\mathbf B'\otimes \mathbf I)$
to each of the above cases.
\begin{itemize}
\item[$\bullet$] (Case 1: \(\nu_1=\nu_2=\nu_3\).)
Then
\[
\mathbf X_0|\nu_1\nu_1\nu_1\rangle
=
\mu_{\nu_1}^2|\nu_1\nu_1\nu_1\rangle.
\]
By part (i) of the previous lemma,
\[
\mathbf X|\nu_1\nu_1\nu_1\rangle
=
\lambda_{\nu_1}\mu_{\nu_1}\,|\nu_1\nu_1\nu_1\rangle.
\]
Hence
\[
\mathbf X\mathbf X_0|\nu_1\nu_1\nu_1\rangle
=
\mu_{\nu_1}^2\lambda_{\nu_1}\mu_{\nu_1}\,|\nu_1\nu_1\nu_1\rangle
=
\mu_{\nu_1}^3\lambda_{\nu_1}\,|\nu_1\nu_1\nu_1\rangle.
\]

\item[$\bullet$] (Case 2: \(\nu_1=\nu_2<\nu_3\).)
Then
\[
\mathbf X_0|\nu_1\nu_1\nu_3\rangle
=
\mu_{\nu_1}\,|\nu_1\nu_3\nu_1\rangle.
\]
Since the first two coordinates satisfy \(\nu_1<\nu_3\), part (i) of the previous lemma gives
\[
\mathbf X|\nu_1\nu_3\nu_1\rangle=0.
\]
Therefore, $\mathbf X\mathbf X_0|\nu_1\nu_1\nu_3\rangle=0.$
\item[$\bullet$] (Case 3: \(\nu_1<\nu_2=\nu_3\).)
Then
\[
\mathbf X_0|\nu_1\nu_2\nu_2\rangle
=
\mu_{\nu_2}\,|\nu_2\nu_2\nu_1\rangle.
\]
Applying part (i) of the previous lemma to \(|\nu_2\nu_2\nu_1\rangle\), we note that
\(\nu_2=\nu_2>\nu_1\), so this falls into the zero case. Hence
\[
\mathbf X|\nu_2\nu_2\nu_1\rangle=0,
\]
and therefore $\mathbf X\mathbf X_0|\nu_1\nu_2\nu_2\rangle=0.$
\item[$\bullet$] (Case 4: \(\nu_1<\nu_2<\nu_3\).)
Then
\[
\mathbf X_0|\nu_1\nu_2\nu_3\rangle
=
|\nu_2\nu_3\nu_1\rangle.
\]
Since \(\nu_2<\nu_3\), part (i) of the previous lemma again gives
\[
\mathbf X|\nu_2\nu_3\nu_1\rangle=0.
\]
Thus $\mathbf X\mathbf X_0|\nu_1\nu_2\nu_3\rangle=0.$
\item[$\bullet$] ({All remaining cases.})
In all other cases,
\[
\mathbf X_0|\nu_1\nu_2\nu_3\rangle=0,
\]
so trivially $\mathbf X\mathbf X_0|\nu_1\nu_2\nu_3\rangle=0.$
\end{itemize}
\end{proof}
\begin{lemma}\label{513ama}
Let $\mathbf{W}_{m,i}$ satisfy
\begin{equation}\label{final}
\mathbf{W}_{m,i}=\mathcal{B}_i\mathbf{W}_{m,i-1}+\mathcal{B}_i'\mathbf{W}_{m,i+1},\qquad i=1,\dots,m.
\end{equation}
Assume that $\mathfrak{A}_k$ defined by
\[
\mathfrak{A}_0=\mathbf{I},\qquad \mathfrak{A}_k=\mathbf{I}-\mathbf{B}_{k+1}\mathfrak{A}_{k-1}^{-1}\mathbf{B}_k'
\]
are invertible. Then
\[
\mathbf{W}_{m,m}=\mathcal{L}_m\mathbf{W}_{m,0}+\mathcal{L}_m'\mathbf{W}_{m,m+1},
\]
where
\[
\mathcal{L}_m=\mathfrak{A}_{m-1}^{-1}\mathcal{B}_m\mathfrak{A}_{m-2}^{-1}\mathcal{B}_{m-1}\cdots \mathfrak{A}_1^{-1}\mathcal{B}_2\mathcal{B}_1.
\]
\end{lemma}

\begin{proof}
We prove by induction on $m$. The case $m=2$ was shown in (\ref{eq:redX-sol}).
Assume the statement holds for $m-1$. From
\begin{equation}
\mathbf{W}_{m,m}=\mathcal{B}_m\mathbf{W}_{m,m-1}+\mathcal{B}_m'\mathbf{W}_{m,m+1},
\end{equation}
by (\ref{final}), and using the induction hypothesis to express $\mathbf{W}_{m,m-1}$ in terms of $\mathbf{W}_{m,0}$ and $\mathbf{W}_{m,m+1}$, we obtain
\[
\mathbf{W}_{m,m}
=
\mathcal{B}_m\mathfrak{A}_{m-2}^{-1}\cdots \mathcal{B}_2\mathcal{B}_1\mathbf{W}_{m,0}
+
\mathcal{B}_m\mathfrak{A}_{m-2}^{-1}\mathcal{B}_{m-1}'\mathbf{W}_{m,m}
+
\mathcal{B}_m'\mathbf{W}_{m,m+1}.
\]
Rearranging yields
\[
(\mathbf{I}-\mathcal{B}_m\mathfrak{A}_{m-2}^{-1}\mathcal{B}_{m-1}')\mathbf{W}_{m,m}
=
\mathcal{B}_m\mathfrak{A}_{m-2}^{-1}\cdots \mathcal{B}_2\mathcal{B}_1\mathbf{W}_{m,0}
+
\mathcal{B}_m'\mathbf{W}_{m,m+1},
\]
which gives the result.
\end{proof}
\section{Proofs}\label{412129am}
\begin{proof}[Proof of Lemma \ref{lem:block-form-Xk}]
We prove the statement by induction on $k$. The case $k=1$ follows from the explicit computation given in Section \ref{412ama}. Assume the statement holds for $k-1$.

With respect to the ordered basis
\[
(e_0,\dots,e_{k-2}\mid e_{k-1},e_k\mid e_{k+1},\dots,e_{n-1}),
\]
the matrix of $\mathbf I - \mathcal X_{k-1}$ has the form
\[
\mathbf I-\mathcal X_{k-1}
=
\left(
\begin{array}{c|cc|c}
\mathbf I-D_{k-1}^- & 0 & 0 & 0\\
\hline
0 & 1 & * & 0\\
0 & 0 & 1 & 0\\
\hline
0 & 0 & 0 & \mathbf I-D_{n-k-1}^+
\end{array}
\right),
\]
where $D_{k-1}^-$ and $D_{n-k-1}^+$ are some diagonal matrices of dimension $k-1$ and $n-k-1$, respectively, and hence
\[
(\mathbf I-\mathcal X_{k-1})^{-1}
=
\left(
\begin{array}{c|cc|c}
\widetilde D_{k-1}^- & 0 & 0 & 0\\
\hline
0 & 1 & * & 0\\
0 & 0 & 1 & 0\\
\hline
0 & 0 & 0 & \widetilde D_{n-k-1}^+
\end{array}
\right),
\]
where $\widetilde D_{k-1}^-$ and $\widetilde D_{n-k-1}^+$ are diagonal matrices of corresponding dimensions.

In particular, $(\mathbf I-\mathcal X_{k-1})^{-1}$ preserves the subspaces
\[
\operatorname{span}\{e_0,\dots,e_{k-2}\},\qquad
\operatorname{span}\{e_{k-1},e_k\},\qquad
\operatorname{span}\{e_{k+1},\dots,e_{n-1}\},
\]
and acts diagonally on the first and third. Now recall that
\[
\mathcal X_k
=
\mathbf{\mathcal B}_{k+1}(\mathbf I-\mathcal X_{k-1})^{-1}\mathbf{\mathcal B}_k'.
\]
We analyze the action of $\mathcal X_k$ on the three subspaces
\[
U_k^-:=\operatorname{span}\{e_0,\dots,e_{k-1}\},\qquad
U_k^0:=\operatorname{span}\{e_k,e_{k+1}\},\qquad
U_k^+:=\operatorname{span}\{e_{k+2},\dots,e_{n-1}\}.
\]

\begin{itemize}
\item[(i)](Action on $U_k^-$)
Let $0\le r\le k-1$. By the local action of $\mathbf{\mathcal B}_k'$, we have
\[
\mathbf{\mathcal B}_k' e_r = a_r e_r
\]
for some scalar $a_r$. Since $U_k^-$ is invariant under $(\mathbf I-\mathcal X_{k-1})^{-1}$, it follows that
\[
(\mathbf I-\mathcal X_{k-1})^{-1}\mathbf{\mathcal B}_k' e_r \in U_k^-.
\]
Moreover, $\mathbf{\mathcal B}_{k+1}$ acts diagonally on $U_k^-$, and hence
\[
\mathcal X_k e_r \in \operatorname{span}\{e_r\}.
\]
Thus, the restriction of $\mathcal X_k$ to $U_k^-$ is diagonal.
\item[(ii)] (Action on $U_k^+$)
Let $r\ge k+2$. Again,
\[
\mathbf{\mathcal B}_k' e_r = b_r e_r
\]
for some scalar $b_r$. Since $(\mathbf I-\mathcal X_{k-1})^{-1}$ acts diagonally on
\[
\operatorname{span}\{e_{k+1},\dots,e_{n-1}\},
\]
we obtain
\[
(\mathbf I-\mathcal X_{k-1})^{-1}\mathbf{\mathcal B}_k' e_r \in \operatorname{span}\{e_r\}.
\]
Applying $\mathbf{\mathcal B}_{k+1}$, which also acts diagonally on $U_k^+$, gives
\[
\mathcal X_k e_r \in \operatorname{span}\{e_r\}.
\]
Thus, the restriction of $\mathcal X_k$ to $U_k^+$ is diagonal.
\item[(iii)] (Action on $U_k^0$.) We consider $U_k^0=\operatorname{span}\{e_k,e_{k+1}\}$. First, $\mathbf{\mathcal B}_k' e_k = 0$, and hence
\[
\mathcal X_k e_k = 0.
\]
Next, write
\[
\mathbf{\mathcal B}_k' e_{k+1}=c_k e_{k+1}
\]
for some scalar $c_k$. Since $e_{k+1}$ lies in the right sector
\[
\operatorname{span}\{e_{k+1},\dots,e_{n-1}\},
\]
there exists a scalar $d_k$ such that
\[
(\mathbf I-\mathcal X_{k-1})^{-1}e_{k+1}=d_k e_{k+1}.
\]
Therefore,
\[
\mathcal X_k e_{k+1}
=
c_k d_k\, \mathbf{\mathcal B}_{k+1}e_{k+1}.
\]
By the local action of $\mathbf{\mathcal B}_{k+1}$, we have
\[
\mathbf{\mathcal B}_{k+1}e_{k+1} \in \operatorname{span}\{e_k\},
\]
and hence
\[
\mathcal X_k e_{k+1} \in \operatorname{span}\{e_k\}.
\]

Thus, with respect to the ordered basis $(e_k,e_{k+1})$, the restriction of $\mathcal X_k$ to $U_k^0$ has the form
\[
\begin{pmatrix}
0 & *\\
0 & 0
\end{pmatrix},
\]
and is therefore nilpotent.
\end{itemize}
\end{proof}

\begin{proof}[Proof of Lemma \ref{1203ama}]
Let \(\alpha:=\mu_1\lambda_1\). On the single-species sector,
\[
\mathcal X_k |1\cdots1\rangle
=
\frac{\alpha}{a_{k-1,1}} |1\cdots1\rangle,
\]
so
\[
c_k=\frac{\alpha}{a_{k-1,1}}.
\]

We argue by induction on \(k\). The case \(k=1\) follows from the explicit form of \(\mathcal X_2\).

Assume that
\[
\mathcal X_k e_r=c_{k-1}e_r
\qquad\text{for }0\le r\le k-1.
\]
Since
\[
c_{k-1}=\frac{\alpha}{a_{k-2,1}},
\]
it follows that
\[
\mathfrak A_k e_r
=
(I-\mathcal X_k)e_r
=
(1-c_{k-1})e_r
=
a_{k-1,1}e_r,
\]
where we used
\[
a_{k-1,1}=1-\frac{\alpha}{a_{k-2,1}}.
\]
Hence
\[
\mathfrak A_k^{-1}e_r=\frac1{a_{k-1,1}}e_r
\qquad (0\le r\le k-1).
\]

For \(0\le r\le k-1\), using
\[
\mathcal X_{k+1}=\mathcal B_{k+2}\mathfrak A_k^{-1}\mathcal B'_{k+1},
\]
together with
\[
\mathcal B'_{k+1}e_r=\lambda_1 e_r,
\qquad
\mathcal B_{k+2}e_r=\mu_1 e_r,
\]
we get
\[
\mathcal X_{k+1}e_r
=
\mathcal B_{k+2}\mathfrak A_k^{-1}(\lambda_1 e_r)
=
\frac{\lambda_1}{a_{k-1,1}}\mathcal B_{k+2}e_r
=
\frac{\mu_1\lambda_1}{a_{k-1,1}}e_r
=
c_ke_r.
\]

For \(r=k\), we use
\[
\mathcal B'_{k+1}e_k=e_{k+1}.
\]
By Lemma~\ref{lem:block-form-Xk},
\[
\mathfrak A_k^{-1}e_{k+1}
=
e_{k+1}+\frac{\lambda_1}{a_{k-1,1}}e_k.
\]
Therefore
\[
\mathcal X_{k+1}e_k
=
\mathcal B_{k+2}\mathfrak A_k^{-1}\mathcal B'_{k+1}e_k
=
\mathcal B_{k+2}\mathfrak A_k^{-1}e_{k+1}
=
\mathcal B_{k+2}\left(e_{k+1}+\frac{\lambda_1}{a_{k-1,1}}e_k\right).
\]
Since
\[
\mathcal B_{k+2}e_{k+1}=0,
\qquad
\mathcal B_{k+2}e_k=\mu_1 e_k,
\]
we obtain
\[
\mathcal X_{k+1}e_k
=
\frac{\lambda_1}{a_{k-1,1}}\mu_1 e_k
=
\frac{\mu_1\lambda_1}{a_{k-1,1}}e_k
=
c_ke_k.
\]

Thus
\[
\mathcal X_{k+1}e_r=c_ke_r
\qquad\text{for all }0\le r\le k,
\]
which completes the induction.
\end{proof}
\section*{Acknowledgments}
The author dedicates this paper to the memory of  his PhD advisor, Craig Tracy, whose guidance and insight had a lasting influence. The author is grateful to Axel Saenz Rodriguez at Oregon State University for valuable discussions during a visit in 2023. This work was supported by Nazarbayev University under the Faculty Development Competitive Research Grants Program for 2024--2026 (Grant No.~201223FD8822, EL).
\bibliographystyle{plain}
\bibliography{refs}
\end{document}